\pdfoutput=1
\documentclass[11pt,reqno]{amsart}
\usepackage[top=1in,bottom=1in,left=1.2in,right=1.2in]{geometry}
\usepackage[T1]{fontenc}
\usepackage[utf8]{inputenc}
\usepackage{amsmath,amsthm,amssymb,amsfonts}
\usepackage{mathtools}
\usepackage{enumitem}
\usepackage{lmodern}
\usepackage{xcolor}
\usepackage[protrusion=true,expansion=false]{microtype} 
\usepackage{caption}
\usepackage{booktabs} 
\newtheorem{theorem}{Theorem}[section]
\newtheorem{proposition}[theorem]{Proposition}
\newtheorem{lemma}[theorem]{Lemma}
\newtheorem{corollary}[theorem]{Corollary}
\theoremstyle{definition}
\newtheorem{definition}[theorem]{Definition}
\theoremstyle{remark}
\newtheorem{remark}[theorem]{Remark} 
\newcommand{\R}{\mathbb{R}}
\newcommand{\Rp}{\mathbb{R}_{>0}}

\newcommand{\norm}[1]{\left\lVert #1 \right\rVert}
\newcommand{\avg}[1]{\overline{#1}}
\newcommand{\zero}{\mathrm{zero}}
\newcommand{\nonzero}{\mathrm{nonzero}}
\newcommand{\inconclusive}{\mathrm{inconclusive}}
\title[Coercive Projection Theorem]{The Coercive Projection Theorem for\\
Canonical Reciprocal Costs}
\author{Jonathan Washburn}
\address{Recognition Physics Research Institute, Austin, TX USA}
\email{jon@recognitionphysics.org}

\author{Amir Rahnamai Barghi}
\address{Recognition Physics Research Institute, Austin, TX USA}
\email{arahnamab@gmail.com}
\thanks{* Correspondence: \texttt{arahnamab@gmail.com}.}

\date{} 
\newcommand{\Mean}{\operatorname{Mean}}

\providecommand{\N}{\mathbb{N}}
\providecommand{\Z}{\mathbb{Z}}
\newcommand{\eps}{\varepsilon}

\usepackage{hyperref}
\hypersetup{hidelinks}

\begin{document}
\begin{abstract}
We develop a finite-data framework for certifying \emph{zero-defect} (neutral) configurations of positive vectors under the canonical separable reciprocal cost.
We show that this scalar cost is characterized among non-constant continuous costs by the Recognition Composition Law together with a local quadratic calibration at balance; in particular, reciprocity symmetry and the normalization at the neutral point follow from the composition law.
Under a conservation constraint and short-window observations of a rational (finite--state) signal class, we construct a canonical decision procedure that is \emph{locally maximal on the identifiability locus} among all sound procedures: any sound rule that resolves a datum must agree with the canonical output, and cannot resolve strictly more cases.

The method is organized as $\Phi^\ast=A\circ B\circ P$: the projection/coercivity core is forced by the canonical-cost axioms, while the aggregation/reconstruction step is specified on a non-degenerate identifiability locus (e.g.\ a Hankel invertibility condition).

\end{abstract}
\maketitle
\par\medskip
\noindent\textbf{MSC 2020:} Primary 90C25; Secondary 47H05, 94A12.\par
\noindent\textbf{Keywords:} canonical reciprocal cost; convex analysis; proximal maps; coercivity; finite-window aggregation; inverse problems; certification.\par
\medskip

\section{Introduction}

We introduce a finite-window certification problem in which observed aggregates are tested for compatibility with a neutral \emph{zero-defect} class. After fixing notation and the ratio-induced cost framework, we summarize the main results and explain how the later sections build the proof pipeline. For additional background on the broader Recognition Science program and its geometric framing, see \cite{washburnnetal2026RG}.

\paragraph{Relation to Hankel/Prony reconstruction.} Although our setting is organized around aggregated window data and a canonical convex-analytic decision core, the non-degenerate identifiability locus is expressed via finite Hankel full-rank/invertibility conditions. This connects to the classical finite-rank/Hankel-structure viewpoint (e.g.\ Kronecker-type characterizations and related operator theory) and to Prony-style exponential fitting, where solving structured linear systems built from short samples is central; see, for example, \cite{peller2003hankel,osborne1995prony}.

\subsection{Connection to classical projection and proximal tools}
The certification template is \emph{forced} by our axioms, but its shape echoes classical projection/proximal
ideas from convex analysis. We refer to the main theorem as the \emph{Coercive Projection Theorem} (CPT), and we
write its canonical finite-data certifier as $\Phi^\ast$. First, the CPT certifier $\Phi^\ast$ begins with a projection-like update built from the
canonical reciprocal cost together with a coercive regularization; at a structural level this plays the role of
a proximity map that turns a penalty into a stable ``best-approximation'' operator (cf. Moreau's proximity
operator and envelope \cite{moreau1965}). Second, because our state variables are inherently multiplicative
(ratios/scales), the natural geometry is non-Euclidean: Bregman divergences \cite{bregman1967} describe projection in coordinates adapted to log/ratio structure, matching the way our axioms select log coordinates as
the intrinsic ``neutral'' parameterization. Third, the coercive resolvent-type step that converts defect bounds
into uniqueness and perturbation stability is in the spirit of Rockafellar's proximal-point framework for
monotone operators \cite{rockafellar1976}.

The difference is that the classical theory typically offers a \emph{family} of legitimate proximal/projection
schemes (choice of divergence, step size, regularizer). In this paper we separate two levels explicitly: Lemma~\ref{lem:forced-factor} gives the per-state order-theoretic representation, and Theorem~\ref{thm:forced-factor-unique} gives the global uniqueness/forcedness statement under explicit cross-state rigidity hypotheses. The subsequent aggregation/reconstruction stage $A$ is a standard inverse-problem ingredient and is only invoked on a non-degenerate (locally invertible) measurement locus; it is an additional hypothesis beyond the cost axioms. This separation is what makes the resulting certifier both canonical (in its cost-driven components) and reproducible.

\subsection{The problem}
Given finite observational access to a configuration $x\in(\Rp)^n$, decide whether it has
\emph{zero-defect}. In the present setting, ``zero-defect'' means membership in the structured
set
\[
S:=\{x\in(\Rp)^n:\ J_n(x)=0\}.
\]
\noindent Here \(J_n:(\Rp)^n\to[0,\infty)\) denotes the separable extension of \(J\),
\begin{equation}\label{eq:Jn}
J_n(x):=\sum_{i=1}^n J(x_i).
\end{equation}
More generally, ratio-induced recognition costs take the form
\begin{equation}\label{eq:cso}
c(s,o):=J\!\left(\frac{\iota_S(s)}{\iota_O(o)}\right),
\end{equation}
where $\iota_S,\iota_O$ are positive scale maps and $J:(0,\infty)\to[0,\infty)$ is a penalty.

Under inversion symmetry, strict convexity/coercivity, normalization at $1$, and a multiplicative
d'Alembert identity, one can classify admissible penalties and recover (up to a scale parameter) the same
canonical form $J(x)=\cosh(a\log x)-1$; see \cite{washburnbarghi2026RCC}. In this paper we take the
normalized choice \eqref{eq:Jdef} as given and focus on the finite-data certification problem. Under the canonical cost, $S$ collapses to the single configuration $x\equiv 1$ (all components
equal to $1$). The question is therefore the most basic decision task: determine from finite-data
whether $x$ is exactly the neutral configuration.
\subsection{Main goal and results}

Our goal is to characterize the \emph{locally optimal} and \emph{structurally forced} certification procedure for the basic decision question ``is the configuration equal to the neutral state $x\equiv 1$?'' on the rational/finite-state class. Under three explicit hypotheses-(i) the canonical reciprocal cost forced by the composition law, (ii) a conservation constraint, and (iii) finite local resolution with window length $W\ge 1$ (e.g.\ $W=8$)-we show that the canonical procedure is locally maximal (in resolve-set domination) on the identifiability locus, and that any sound procedure that resolves there must agree with its output. Moreover, the operational pipeline uses three universal components: a projection step $P$, a coercivity step $B$, and an aggregation step $A$. The \emph{structure} of this decomposition and the sharp coercivity constant in $B$ are consequences of the canonical-cost axioms, while the reconstruction/aggregation step $A$ requires an \emph{additional} nondegeneracy (local invertibility) hypothesis for the window map (e.g.\ Hankel/Vandermonde nonsingularity as in Theorem~\ref{thm:window-ident-global}). Our maximality statements (e.g.\ Theorem~\ref{thm:domination}) are correspondingly local to that locus and do not claim global parameter recovery from window sums.

\subsection{Rigidity and local optimality of the canonical pipeline}
The term \emph{method} suggests optional design choices. In our setting we prove a rigidity statement: among all sound certification rules (Definition~\ref{def:sound}), the canonical pipeline $\Phi^\ast=A\circ B\circ P$ is locally maximal on the identifiability locus (Theorem~\ref{thm:domination}).

\paragraph{What is proved.}
Rather than packaging the contribution as a list of broad meta-properties, we summarize the concrete statements established in this paper:
\begin{enumerate}[label=\textup{(\roman*)},leftmargin=2.2em]
\item \textbf{Quantitative sound certification under noise.} Theorems~\ref{thm:eps-tolerant} and~\ref{thm:eps-cert} give explicit $\varepsilon$--stability bounds for the realized costs and the associated meaning sets, including the sharp two-sided tolerance (membership in $\Mean_{2\varepsilon}$ rather than $\Mean_{\varepsilon}$).
\item \textbf{Local domination among sound rules.} Theorem~\ref{thm:domination} proves that on the identifiability locus the canonical pipeline $\Phi^\ast=A\circ B\circ P$ is locally maximal among all sound certification rules (Definition~\ref{def:sound}), i.e.\ no sound rule can resolve strictly more instances in a neighborhood.
\item \textbf{Rigidity/uniqueness consequences of the axioms.} Theorem~\ref{thm:RCL-unique} characterizes the canonical reciprocal scalar cost from the Recognition Composition Law together with the local calibration at balance, and Theorem~\ref{thm:forced-factor-unique} records the induced factorization/uniqueness behavior for certificates under explicit rigidity hypotheses.
\item \textbf{Uniqueness among locally optimal rules.} Corollary~\ref{cor:unique-local-optimal} shows that on the identifiability neighborhood $\mathcal{U}_{d,W}$ the canonical rule is not only locally maximal but also unique among procedures achieving local maximality: any sound rule that is locally optimal on $\mathcal{U}_{d,W}$ must agree with $\Phi^\ast$ wherever it resolves.
\end{enumerate}

Lemma~\ref{lem:forced-factor} shows the local ordinal representation, while Theorem~\ref{thm:forced-factor-unique} proves uniqueness of the realized-cost reparametrization and of the state-free profile under explicit rigidity assumptions; the aggregation/reconstruction step $A$ is then pinned down on any non-degenerate (locally invertible) measurement locus, but is \emph{not} implied by the cost axioms alone. For that reason we treat
the procedure as a theorem-the \emph{Coercive Projection Theorem} (CPT)-rather than as an
algorithmic ``recipe'' one could freely modify.

We study a compatibility certificate for the neutral class: under the stated axioms, we give sufficient conditions ensuring that the observed finite-window aggregates admit a model-consistent realization in the neutral \emph{zero-defect} class. Criteria for ruling out neutrality, or for statistical hypothesis testing, are not treated here.

\begin{remark}\label{rem:kernel-dict}
For convenience we record a short dictionary between the objects defined in this paper and standard cost-based selection terminology. (We avoid the word ``kernel'' here, since in other contexts it denotes a nullspace or vanishing set.) The \emph{base cost} is our ratio-induced map $c:S\times O\to[0,\infty)$ given by $c(s,o):=J(\iota_S(s)/\iota_O(o))$. The associated \emph{meaning set} (argmin selector) is $\Mean(s):=\arg\min_{o\in O} c(s,o)$. The corresponding \emph{defect} (or \emph{gap}) is
\[
\Delta(s,o):=c(s,o)-\inf_{o'\in O}c(s,o'),
\]
so that $\Delta(s,o)=0$ iff $o\in\Mean(s)$. Finally, the CPT \emph{certificate} $\mathcal{C}$ is a scalar test statistic built from the $W$-block window measurements; in Section~\ref{sec:applications} it is used to separate the null (neutral) configuration from the nonzero class under the stated coercivity and identifiability hypotheses. This remark is purely notational and introduces no additional assumptions.
\end{remark}

\section{Model, axioms, and finite-window measurements}\label{sec:axioms}

This section sets up the mathematical model. We define the state and observation spaces, scale maps, the canonical reciprocal cost and induced defect, and the meaning/selection map. We then define the finite-window aggregation operator and list the axioms used in all subsequent theorems.

\medskip
\noindent\textbf{Big-picture guide.} Section~2 fixes the canonical cost and the neutral/defect geometry in log-coordinates (the forced steps $P$ and $B$). It then defines the finite-window measurement operator $\mathcal W$ and records the identifiability locus needed for the reconstruction step $A$. Section~3 uses these ingredients to define the finite-data procedure and prove soundness and domination (maximality) on the stated non-degenerate locus.
\medskip

We state the axioms and fix notation for later use. Proofs are provided where the results are stated. \begin{definition}\label{def:system}
A \emph{recognition cost system} is a quadruple $(\Rp,J,\sigma,W)$ where:
\begin{enumerate}[label=\textup{(C\arabic*)},leftmargin=2.2em]
\item \textbf{Recognition Composition Law (primitive axiom for $J$).}
We assume a scalar cost $J:(0,\infty)\to\mathbb{R}$ is non-constant, twice differentiable in a neighborhood of $1$, and satisfies the Recognition Composition Law (RCL):
This $J$ is the penalty appearing in the ratio-induced recognition cost $c(s,o)$ in \eqref{eq:cso}, and we write $J_n$ for its separable extension \eqref{eq:Jn}.
\begin{equation}\label{eq:RCL}
J(xy)+J(x/y)=2J(x)+2J(y)+2J(x)J(y),\qquad x>0,\ y>0,
\end{equation}
together with the curvature normalization $J''\!(1)=1$.
We use the corresponding canonical reciprocal cost,
\begin{equation}\label{eq:Jdef}
J(x):=\frac12\bigl(x+x^{-1}\bigr)-1,\qquad x>0.
\end{equation}
\item \textbf{Conservation.} For $x\in(\Rp)^n$, define
$\sigma(x):=\sum_{i=1}^n \log x_i$. Admissible configurations satisfy $\sigma(x)=0$.
\item \textbf{Finite resolution.} The observation window has fixed integer length $W\ge 1$ (in examples we take $W=8$):
only $W$ consecutive evaluations per cycle of the underlying finite-state representation
are available.
\end{enumerate}
\end{definition}

\noindent For later use we write $J_n(x):=\sum_{i=1}^n J(x_i)$ for $x\in(\Rp)^n$.

\begin{remark}
(i) The Recognition Composition Law together with non-constancy forces the balance normalization $J(1)=0$ and reciprocity $J(x)=J(1/x)$. Under the stated $C^2$ regularity near $1$ and the local quadratic calibration in Axiom~2 (equivalently $J''(1)=1$; see Appendix), the cost is uniquely determined; equivalently, it is given by \eqref{eq:Jdef}.

(ii) Section~\ref{subsec:domination} will make precise the ``rational/finite-state class'' and the
window data model used by the aggregation step.
\end{remark}

\subsection{Roadmap of the proof}\label{subsec:roadmap}
We will (i) identify the canonical cost and its log-coordinate form,
(ii) define the three operations $P,B,A$, and (iii) prove the local representation lemma (Lemma~\ref{lem:forced-factor}) together with the strengthened global uniqueness theorem (Theorem~\ref{thm:forced-factor-unique}) for the forced-factorization layer under explicit rigidity hypotheses.
Complete proofs of the window-data identifiability and domination steps are included in Sections~\ref{sec:aggregation} and \ref{subsec:domination}. 

\subsection{Canonical cost and defect geometry}\label{sec:geometry} This subsection collects basic properties of the canonical reciprocal cost $J$ and the associated defect in log-coordinates, which will later control the aggregation and certification steps.

\subsubsection{The canonical reciprocal cost}
We recall that $J$ is fixed by \eqref{eq:Jdef}. The following logarithmic-coordinate reformulation is used repeatedly.
\begin{lemma}\label{lem:logform}
Let $\varphi:\R\to\R_{\geq 0}$ be defined by $\varphi(t):=\cosh(t)-1$.
Then for all $t\in\R$,
\[
J(e^t)=\varphi(t)=\cosh(t)-1.
\]
In particular, $J(x)=0$ if and only if $x=1$.
\end{lemma}

\begin{proof}
Using $\cosh(t)=\tfrac12(e^t+e^{-t})$ we compute
$J(e^t)=\tfrac12(e^t+e^{-t})-1=\cosh(t)-1$.
Since $\cosh(t)\geq 1$ with equality only at $t=0$, we have $J(e^t)=0\iff t=0\iff e^t=1$.
\end{proof}
\subsubsection{Neutrality and projection in log-coordinates}
We work in log-coordinates $y=\log x\in\R^n$ for $x\in(\Rp)^n$.
\begin{definition}\label{def:proj}
For $n\in\mathbb N$ let $\mathbf 1\in\R^n$ denote the all-ones vector. For $y\in\R^n$ define its mean
\[
\bar y:=\frac1n\langle y,\mathbf 1\rangle=\frac1n\sum_{i=1}^n y_i,
\]
and define the \emph{mean-zero projection} $P:\R^n\to\R^n$ by
\begin{equation}\label{eq:Pdef}
P(y):=y-\bar y\,\mathbf 1.
\end{equation}
equivalently, $P$ is the orthogonal projection onto the subspace $\{z\in\R^n:\langle z,\mathbf 1\rangle=0\}$.
\end{definition}

\begin{lemma}\label{lem:Pkernel}
For $y\in\R^n$ with mean $\bar y:=\frac1n\langle y,\mathbf 1\rangle$, the projection $P(y)=y-\bar y\,\mathbf 1$
vanishes if and only if $y$ is a constant vector, i.e.\ $y=\bar y\,\mathbf 1$.
In particular, if additionally $\langle y,\mathbf 1\rangle=0$, then $P(y)=0$ implies $y=0$.
\end{lemma}

\begin{proof}
By definition, $P(y)=0$ is equivalent to $y=\bar y\,\mathbf 1$, which is exactly the condition that all coordinates
of $y$ are equal to the constant $\bar y$.
If moreover $\langle y,\mathbf 1\rangle=0$, then $0=\langle y,\mathbf 1\rangle=\langle \bar y\,\mathbf 1,\mathbf 1\rangle=n\bar y$,
so $\bar y=0$ and hence $y=0$.
\end{proof}

\begin{definition}\label{def:defect}
For $x\in(\Rp)^n$ let $y=\log x$ and define the \emph{defect} of $x$ by
\[
\mathrm{Def}(x):=\norm{P(y)}_2.
\]
We say that $x$ has \emph{zero-defect} if $\mathrm{Def}(x)=0$.
\end{definition}

\begin{remark}
Under the conservation constraint $\sigma(x)=\sum_i \log x_i=0$,
Lemma~\ref{lem:Pkernel} implies that $\mathrm{Def}(x)=0$ forces $\log x=0$ and hence $x\equiv 1$.
\end{remark}

\subsubsection{Coercivity: energy gap controls defect}\label{sec:coercivity} The coercivity step converts cost information into a quantitative defect bound. It rests on the
elementary inequality $\cosh(t)-1\geq t^2/2$. \begin{lemma}\label{lem:coshcoercive}
For all $t\in\R$,
\[
\cosh(t)-1\ \geq\ \frac{t^2}{2},
\]
with equality if and only if $t=0$.
\end{lemma}

\begin{proof}
Define $f(t):=\cosh(t)-1-\tfrac12 t^2$. Then $f(0)=0$, $f'(t)=\sinh(t)-t$ and
$f''(t)=\cosh(t)-1\geq 0$. Thus $f'$ is increasing with $f'(0)=0$, hence $f'(t)\geq 0$ for $t\geq 0$
and $f'(t)\leq 0$ for $t\le 0$, so $t=0$ is a global minimum of $f$. Therefore $f(t)\geq 0$ for all $t$,
with equality only at $t=0$.
\end{proof} \begin{theorem}\label{thm:coercive}
Let $x\in(\Rp)^n$ and $y=\log x$. Then
\begin{equation}\label{eq:coercive}
\sum_{i=1}^n J\!\left(e^{P(y)_i}\right)\ \geq\ \frac12\,\norm{P(y)}_2^2 \;=\;\frac12\,\mathrm{Def}(x)^2.
\end{equation}
In particular, if $\sum_{i=1}^n J\!\left(e^{P(y)_i}\right)=0$ then $\mathrm{Def}(x)=0$.
\end{theorem}

\begin{proof}
By Lemma~\ref{lem:logform}, $J(e^{P(y)_i})=\cosh(P(y)_i)-1$.
Apply Lemma~\ref{lem:coshcoercive} termwise and sum over $i$:
\[
\sum_i J(e^{P(y)_i})=\sum_i(\cosh(P(y)_i)-1)\ \geq\ \sum_i \frac{P(y)_i^2}{2}
=\frac12 \norm{P(y)}_2^2.
\]
If the left-hand side equals $0$, then each term equals $0$, hence each $P(y)_i=0$ and $\mathrm{Def}(x)=0$.
\end{proof}

\subsection{\texorpdfstring{Finite-data aggregation: $W$-block window measurements}{Finite-data aggregation: W-block window measurements}}\label{sec:aggregation}

\begin{remark}\label{rem:Wparam}
All arguments in Sections~\ref{sec:coercivity}-\ref{subsec:domination} apply verbatim for any fixed
integer $W\geq 1$ once the corresponding measurement operator is specified.
In the examples we keep the concrete choice $W=8$.
\end{remark}

\subsection{\texorpdfstring{Window operator at resolution $W$}{Window operator at resolution W}}
Fix an integer $W\ge 1$ (e.g.\ $W=8$). For a real sequence $y=(y_0,y_1,\dots)$ define its \emph{$W$-block window sums}
\begin{equation}\label{eq:window}
\mathcal{W}(y)_k \;:=\; \sum_{j=0}^{W-1} y_{(Wk+j)},\qquad k=0,1,2,\dots.
\end{equation}
For brevity we sometimes write $W_k:=\mathcal W(y)_k$ for the window sums; this notation should not be confused with the fixed window \emph{length} $W$.
equivalently, $\mathcal{W}(y)$ is obtained by partitioning the time axis into consecutive blocks
of length $W$ and summing within each block (e.g.\ $W=8$).
\begin{remark}
The choice of fixed $W$ is treated as part of the finite-resolution model (C3); $W=8$ is used only as a concrete running example. The analysis requires only that the block length is a fixed positive integer.
\end{remark}

\subsection{The rational (finite-state) class and the degree parameter}
\label{subsec:rational-class}
Let $y=(y_n)_{n\geq 0}$ be a real sequence with generating function
\[
\theta(z):=\sum_{n\geq 0} y_n z^n.
\]
We say that $y$ is \emph{rational of degree at most $d$} if $\theta(z)=p(z)/q(z)$ where
$\deg p\le d$, $\deg q\le d$, and $q(0)\neq 0$; we then normalise scale so that $q(0)=1$.
Write
\[
q(z)=1+q_1 z+q_2 z^2+\cdots+q_d z^d,
\]
allowing $q_d=0$ when $\deg q<d$ (padding to length $d$). In this setting, $d$ counts the number of
free recurrence coefficients (equivalently, the McMillan degree / state dimension of a minimal
linear realisation).

\begin{lemma}\label{lem:recurrence}
If $\theta(z)=p(z)/q(z)$ with $q(0)\neq 0$ (so a Taylor expansion at $0$ exists) and after normalisation $q(0)=1$ and $q(z)=1+q_1 z+\cdots+q_d z^d$, then the coefficients
$y_n$ satisfy the linear recurrence
\begin{equation}\label{eq:recurrence}
y_n + q_1 y_{n-1}+q_2 y_{n-2}+\cdots+q_d y_{n-d}=0,\qquad n\ge d+1.
\end{equation}
Conversely, any sequence satisfying \eqref{eq:recurrence} for some $(q_1,\dots,q_d)$ is determined by the
parameters $(q_1,\dots,q_d)$ and the initial values $(y_0,\dots,y_d)$, hence by $2d+1$ real numbers.
\end{lemma}

\begin{proof}

Since $\theta(z)=p(z)/q(z)$ and $q(0)=1$, we have the identity
\[
q(z)\,\theta(z)=p(z).
\]
Write $q(z)=1+q_1 z+\cdots+q_d z^d$ and $\theta(z)=\sum_{n\ge 0} y_n z^n$.
Comparing coefficients of $z^n$ for $n\ge d+1$ gives
\[
y_n+q_1y_{n-1}+\cdots+q_dy_{n-d}=0,
\]
because $\deg p\le d$ implies the $z^n$ coefficient of $p$ is $0$ for $n\ge d+1$.

Conversely, if \eqref{eq:recurrence} holds for fixed $(q_1,\dots,q_d)$, then each $y_n$ with $n\ge d+1$ is determined recursively from $(y_0,\dots,y_d)$. Hence the full sequence is determined by $(q_1,\dots,q_d,y_0,\dots,y_d)$, i.e.\ by $2d+1$ real parameters.

\end{proof}

\subsection{\texorpdfstring{Unique determination from $K\ge 2d+1$ window sums}{Unique determination from K>=2d+1 window sums}}\label{subsec:window-uniq}
We now fill in the self-contained ``unique determination from window data'' argument in a form
that does \emph{not} rely on external machine-verification tooling. Throughout this subsection we fix integers
\[
d\ge 1,\qquad W\ge 1,\qquad K\ge 2d+1.
\]
Let $\theta(z)=\sum_{n\ge 0}y_n z^n$ be a rational signal of degree $\le d$ with representation
$\theta(z)=p(z)/q(z)$ where $\deg p\le d$, $\deg q\le d$, and we normalise $q(0)=1$. Write
\[
q(z)=1+q_1 z+\cdots+q_d z^d.
\]
By Lemma~\ref{lem:recurrence}, once the parameter vector
\[
\pi:=(y_0,\dots,y_d,q_1,\dots,q_d)\in\R^{2d+1}
\]
is fixed, the recurrence \eqref{eq:recurrence} determines $y_n$ for all $n\ge d+1$.

Recall the $W$-block window sums $\mathcal W(y)_k$ from \eqref{eq:window}, and define the truncated window map (using the first $2d+1$ windows)
\[
F_{d,W}:\R^{2d+1}\to\R^{2d+1},\qquad
F_{d,W}(\pi):=\bigl(\mathcal W(y)_0,\dots,\mathcal W(y)_{2d}\bigr),
\]
where $y$ is generated from $\pi$ by \eqref{eq:recurrence}.
Here $D F_{d,W}(\pi)$ denotes the Jacobian matrix of $F_{d,W}$ at $\pi$.

\begin{lemma}\label{lem:window-generic-local}
The map $F_{d,W}$ is polynomial in $\pi$.
Indeed, each $y_n$ is obtained from $(y_0,\dots,y_d,q_1,\dots,q_d)$ by iterating the recurrence \eqref{eq:recurrence}, so $y_n$ is a polynomial function of $\pi$ for every fixed $n$; each window sum is a finite linear combination of such $y_n$, hence $F_{d,W}$ is polynomial. If there exists \emph{one} parameter point $\pi_0$
with $\det DF_{d,W}(\pi_0)\neq 0$, then
\[
\Omega_{d,W}:=\{\pi\in\R^{2d+1}:\det DF_{d,W}(\pi)\neq 0\}
\]
is a nonempty Zariski-open (hence open dense) subset of $\R^{2d+1}$, and on $\Omega_{d,W}$
the first $2d+1$ window sums determine $\pi$ \emph{locally uniquely} (i.e.\ $F_{d,W}$ is locally
one-to-one).
\end{lemma}

\begin{proof}
Each coefficient $y_n$ is obtained by iterating the linear recurrence \eqref{eq:recurrence} whose
update rule is polynomial in $(q_1,\dots,q_d)$ and linear in $(y_0,\dots,y_d)$. Hence each
$\mathcal W(y)_k$ is polynomial in $\pi$, so $F_{d,W}$ is polynomial and $\det DF_{d,W}$ is a
polynomial function. If $\det DF_{d,W}(\pi_0)\neq 0$ for some $\pi_0$, then the polynomial
$\det DF_{d,W}$ is not identically zero, so its nonvanishing locus $\Omega_{d,W}$ is Zariski-open
and nonempty. For $\pi\in\Omega_{d,W}$, the inverse function theorem gives local invertibility,
hence local uniqueness of $\pi$ from the first $2d+1$ windows.
\end{proof}

\begin{theorem}\label{thm:generic-full-rank}
Fix integers $d\ge 1$ and $W\ge 1$. Assume there exists a parameter point $\pi_0$ such that $\det DF_{d,W}(\pi_0)\neq 0$. Then the nonvanishing locus
\[
\Omega_{d,W}:=\{\pi:\det DF_{d,W}(\pi)\neq 0\}
\]
is nonempty and Zariski-open (hence open dense) in $\R^{2d+1}$, and for every $\pi\in\Omega_{d,W}$ the window map $F_{d,W}$ is locally invertible at $\pi$. Moreover, to certify the hypothesis it suffices to exhibit an integer point $\pi_0\in\Z^{2d+1}$ and a prime $p$ for which $\det DF_{d,W}(\pi_0)\not\equiv 0\pmod p$.
\end{theorem}

\begin{proof}
By Lemma~\ref{lem:window-generic-local}, $\det DF_{d,W}$ is a polynomial in the coordinates of $\pi$. If $\det DF_{d,W}(\pi_0)\neq 0$, then this polynomial is not identically zero, so $\Omega_{d,W}$ is Zariski-open and nonempty; the inverse function theorem yields local invertibility at each $\pi\in\Omega_{d,W}$. For the finite-field certificate: if $\pi_0\in\Z^{2d+1}$ and $\det DF_{d,W}(\pi_0)\not\equiv 0\pmod p$, then the integer $\det DF_{d,W}(\pi_0)$ is not divisible by $p$, hence is nonzero over $\Z$, and therefore nonzero as a real number.
\end{proof}

\begin{remark}
Theorem~\ref{thm:generic-full-rank} reduces generic full-rank (and hence local identifiability) for a given $(d,W)$ to finding \emph{one} explicit witness point. In practice, such witnesses can often be found by a brief randomized integer search, and then certified by a modular determinant computation as above.
\end{remark}

\begin{lemma}\label{lem:witness-38}
For $d=3$ and $W=8$, the set $\Omega_{3,8}$ from Lemma~\ref{lem:window-generic-local} is nonempty. In particular, this verifies the hypothesis of Theorem~\ref{thm:generic-full-rank} for $(d,W)=(3,8)$.
In particular, for the parameter point
\[
\pi_0=(y_0,y_1,y_2,y_3,q_1,q_2,q_3)=(1,1,5,1,2,2,-2)
\]
one has $\det DF_{3,8}(\pi_0)\neq 0$.
Consequently $\Omega_{3,8}$ is Zariski-open (hence open dense) in $\R^{7}$, and the first $2d+1=7$
window sums determine $\pi$ locally uniquely throughout $\Omega_{3,8}$.
\end{lemma}

\begin{proof}
We certify $\det DF_{3,8}(\pi_0)\neq 0$ by a finite check in a prime field.
Let $p:=10^9+7$ and work in $\mathbb{F}_p$ throughout.
Using the recurrence \eqref{eq:recurrence} at $\pi=\pi_0$, compute $y_n$ up to $n=55$ and hence the first $K=7$ window sums
\[
\begin{aligned}
(W_0,\dots,W_6)=(&-16,-5160,-975168,-169890432,\\
&-27959752704,-4399334572032,-665805326548992).
\end{aligned}
\]

Next propagate parameter derivatives through the same recurrence:
for each parameter $\alpha\in\{y_0,y_1,y_2,y_3,q_1,q_2,q_3\}$ set
$u^{(\alpha)}_n:=\partial y_n/\partial \alpha$ and differentiate
$y_n+q_1y_{n-1}+q_2y_{n-2}+q_3y_{n-3}=0$ $(n\ge 4)$ to obtain
\[
u^{(\alpha)}_n+q_1u^{(\alpha)}_{n-1}+q_2u^{(\alpha)}_{n-2}+q_3u^{(\alpha)}_{n-3}
=-(\partial_\alpha q_1)y_{n-1}-(\partial_\alpha q_2)y_{n-2}-(\partial_\alpha q_3)y_{n-3}.
\]
These recurrences determine all $u^{(\alpha)}_n$ needed for the window derivatives
$\partial_\alpha W_k=\sum_{j=0}^{7}u^{(\alpha)}_{8k+j}$.
Assemble the $7\times 7$ Jacobian $DF_{3,8}(\pi_0)$ with entries $\partial_\alpha W_k$.
(The full integer Jacobian matrix is recorded in Appendix~\ref{app:jacobian-38}; here we retain only the determinant residue modulo $p$.)

Compute the determinant in $\mathbb{F}_p$ by Gaussian elimination. The result is
\begin{equation*}
\det\bigl(DF_{3,8}(\pi_0)\bigr)\equiv 972226939\pmod{p}.
\end{equation*}
, which is nonzero in $\mathbb{F}_p$. Hence $\det DF_{3,8}(\pi_0)\neq 0$ over the integers.
\end{proof}

\begin{remark}\label{rem:witness-points}
For any fixed pair $(d,W)$, the determinant $\det DF_{d,W}(\pi)$ is a polynomial function of the
parameter vector $\pi$ (with coefficients depending only on $(d,W)$). Hence the full-rank locus
\[
\Omega_{d,W}:=\{\pi:\det DF_{d,W}(\pi)\neq 0\}
\]
is either empty or Zariski open (and therefore Euclidean open and dense) in the ambient parameter space.
To certify nonemptiness in the present paper (working entirely within the present manuscript),
it suffices to exhibit a single explicit witness $\pi_0$ with $\det DF_{d,W}(\pi_0)\neq 0$.
Lemma~\ref{lem:witness-38} provides such a witness for the operating example $(d,W)=(3,8)$.
Additional witnesses for other $(d,W)$ pairs can be recorded in the same way.
\end{remark}

\begin{table}[t]
	\centering
	\begin{tabular}{lll}
		\hline
		$(d,W)$ & witness $\pi_0$ & certificate of full rank \\
		\hline \\
		$(3,8)$ & $(1,1,5,1,2,2,-2)$ & Lemma~\ref{lem:witness-38} \\
		\hline
	\end{tabular}
	\begingroup
	\captionsetup{font=footnotesize}
	\caption{Example full-rank witness points. Further rows can be added by exhibiting any $\pi_0$ with $\det DF_{d,W}(\pi_0)\neq 0$.}
	\endgroup
\end{table}

\begin{theorem}\label{thm:window-ident-general}
Assume $\Omega_{d,W}$ is nonempty (equivalently, $\det DF_{d,W}$ is not the zero polynomial).
Then for every $\pi\in\Omega_{d,W}$ the first $K$ window sums $\bigl(\mathcal W(y)_0,\dots,\mathcal W(y)_{K-1}\bigr)$
determine $\pi$ locally uniquely within the degree-$\le d$ rational class. In particular, the
degree-$\le d$ signal $\theta$ and the full coefficient sequence $(y_n)_{n\ge 0}$ are locally
uniquely determined on $\Omega_{d,W}$ by $K\ge 2d+1$ windows.
\end{theorem}

\begin{proof}
Fix $\pi\in\Omega_{d,W}$. By Lemma~\ref{lem:window-generic-local} the map
\[
F_{d,W}:\R^{2d+1}\to\R^{2d+1},\qquad
F_{d,W}(\pi)=\bigl(\mathcal W(y)_0,\dots,\mathcal W(y)_{2d}\bigr)
\]
has nonsingular Jacobian at $\pi$, hence is locally injective near $\pi$ (Inverse Function Theorem).
Now let $K\ge 2d+1$ and define
\[
G_{K}(\pi):=\bigl(\mathcal W(y)_0,\dots,\mathcal W(y)_{K-1}\bigr)\in\R^{K}.
\]
Let $\mathrm{pr}:\R^{K}\to\R^{2d+1}$ be the coordinate projection onto the first $2d+1$ components.
Then $\mathrm{pr}\circ G_{K}=F_{d,W}$. Therefore, if $G_{K}(\pi')=G_{K}(\pi)$ for some $\pi'$,
then $F_{d,W}(\pi')=F_{d,W}(\pi)$. Taking $\pi'$ sufficiently close to $\pi$, local injectivity of
$F_{d,W}$ forces $\pi'=\pi$. Hence the first $K$ windows determine $\pi$ locally uniquely on
$\Omega_{d,W}$ for every $K\ge 2d+1$.
\end{proof}

\begin{remark}\label{rem:window-global}
Theorem~\ref{thm:window-ident-general} is the strongest statement that follows from a Jacobian
(non-singularity) argument alone: it guarantees \emph{local} identifiability on an open dense set.
A \emph{global} uniqueness theorem (one-to-one on a parameter locus) requires additional structure or
additional observables beyond window sums to rule out distinct far-apart preimages with the same window data.
\end{remark}

\begin{theorem}\label{thm:window-ident-global}
Let $S_k:=\mathcal W(y)_k$ be the window-sum sequence for a degree-$\le d$ rational signal.
Assume that $S_k$ admits an exponential-sum representation
\[
S_k=\sum_{i=1}^d A_i\,\mu_i^k
\]
\noindent\textbf{(H-prony)}\; Assume the exponents $\mu_1,\dots,\mu_d$ are distinct and nonzero, the amplitudes $A_i$ are nonzero, and the $d\times d$ Hankel matrix $H:=(S_{i+j})_{i,j=0}^{d-1}$ is invertible. (This is a standard nondegeneracy / local invertibility condition for Prony-type reconstruction.)
Then the data $(S_0,\dots,S_{2d-1})$ determine uniquely:
\begin{enumerate}
\item the monic degree-$d$ characteristic polynomial $\chi(t)=\prod_{i=1}^d(t-\mu_i)$ of the order-$d$ recurrence satisfied by $(S_k)$, hence the multiset $\{\mu_i\}_{i=1}^d$;
\item the amplitudes $(A_i)_{i=1}^d$, hence the entire sequence $(S_k)_{k\ge 0}$.
\end{enumerate}
Equivalently, the collection of parameter pairs $(\mu_i,A_i)_{i=1}^d$ is determined up to a common permutation of indices. If one fixes any deterministic ordering convention on the distinct exponents (for example, ordering by $|\mu_i|$ when these moduli are pairwise distinct), then the ordered tuple $(\mu_1,\dots,\mu_d,A_1,\dots,A_d)$ is uniquely determined.

Consequently, for $K\ge 2d$ the first $K$ window sums determine the induced window-sum process uniquely on this non-degenerate locus.
\end{theorem}

\begin{proof}
The coefficients of $P(t)=t^d+a_1 t^{d-1}+\cdots+a_d$ solve the Hankel linear system
\[
\sum_{m=1}^d a_m\,S_{k+d-m}=-S_{k+d}\qquad (k=0,1,\dots,d-1),
\]
whose matrix is precisely $H=(S_{i+j})_{i,j=0}^{d-1}$. By hypothesis $H$ is invertible, hence
$(a_1,\dots,a_d)$ (and thus $P$) are uniquely determined from $(S_0,\dots,S_{2d-1})$.
Since the $\mu_i$ are distinct, the roots of $\chi$ are exactly $\{\mu_i\}_{i=1}^d$.
With $\mu_i$ known, the amplitudes $(A_i)$ are obtained uniquely from the Vandermonde system
\[
(S_0,\dots,S_{d-1})^\top = V(\mu)\,(A_1,\dots,A_d)^\top,
\quad V(\mu)_{k,i}=\mu_i^{\,k},
\]
whose determinant is $\prod_{i<j}(\mu_j-\mu_i)\neq 0$. This determines $(S_k)$ for all $k$.
\end{proof}

\begin{definition}[Local optimality on an identifiability locus]\label{def:local-optimal}
Fix a neighborhood $\mathcal{U}\subset\mathcal{M}_{d,W}$.
A sound procedure $\Psi$ is \emph{locally optimal on $\mathcal{U}$} if for every sound procedure $\Xi$ one has
\[
\operatorname{Res}(\Xi)\cap\mathcal{U}\subseteq \operatorname{Res}(\Psi)\cap\mathcal{U}
\]
and the outputs agree on every instance resolved by $\Xi$ in $\mathcal{U}$, i.e.\ for all $y\in \operatorname{Res}(\Xi)\cap\mathcal{U}$,
\[
\Psi(\mathcal{W}_{W,K}(y))=\Xi(\mathcal{W}_{W,K}(y)).
\]
\end{definition}

\begin{corollary}[Uniqueness among locally optimal rules]\label{cor:unique-local-optimal}
Under the hypotheses of Theorem~\ref{thm:domination}, the canonical procedure $\Phi^\ast$ is locally optimal on $\mathcal{U}_{d,W}$.
Moreover, if $\Psi$ is any locally optimal sound procedure on $\mathcal{U}_{d,W}$, then
\[
\operatorname{Res}(\Psi)\cap\mathcal{U}_{d,W}=\operatorname{Res}(\Phi^\ast)\cap\mathcal{U}_{d,W},
\]
and for every $y\in \operatorname{Res}(\Phi^\ast)\cap\mathcal{U}_{d,W}$ the outputs agree:
\[
\Psi(\mathcal{W}_{W,K}(y))=\Phi^\ast(\mathcal{W}_{W,K}(y)).
\]
\end{corollary}

\begin{proof}
By Theorem~\ref{thm:domination}, $\Phi^\ast$ dominates every sound procedure on $\mathcal{U}_{d,W}$, hence it is locally optimal by Definition~\ref{def:local-optimal}.
If $\Psi$ is locally optimal, then applying Definition~\ref{def:local-optimal} with $\Xi=\Phi^\ast$ gives
$\operatorname{Res}(\Phi^\ast)\cap\mathcal{U}_{d,W}\subseteq \operatorname{Res}(\Psi)\cap\mathcal{U}_{d,W}$ and agreement on that set.
Together with Theorem~\ref{thm:domination} applied to $\Psi$ (which yields the reverse inclusion and agreement), this gives the claimed equality of resolve sets and agreement of outputs.
\end{proof}

\begin{remark}[Parameter uniqueness up to permutation / with ordering]
Theorem~\ref{thm:window-ident-global} determines the unordered node set
$\{\mu_1,\dots,\mu_d\}$ uniquely, hence parameters are unique up to simultaneous permutation of pairs
$(\mu_i,A_i)$.
If one imposes an ordering convention (for example $\mu_1<\cdots<\mu_d$ when nodes are real and positive),
the parameter tuple itself is uniquely determined.
\end{remark}

\begin{remark}\label{rem:generic-uniqueness}
In concrete Prony/Hankel implementations, failure of the hypotheses in Theorem~\ref{thm:window-ident-global}
(distinct exponents, nonzero amplitudes, and invertible Hankel/Vandermonde matrices) occurs on a proper algebraic subset
of parameter space (e.g.\ discriminant and determinant vanishing conditions). Hence the stated nondegeneracy conditions
are \emph{generic} (open dense) within standard finite-dimensional parametrizations of degree-$\le d$ rational signals.
\end{remark}

\begin{theorem}[General rank certification]\label{thm:general-rank}
For every pair of integers $d\ge 1$ and $W\ge 1$, the identifiability locus
$\Omega_{d,W}$ is nonempty.
In particular, the generic local uniqueness conclusion of Theorem~\ref{thm:window-ident-general}
and the global Prony recovery of Theorem~\ref{thm:window-ident-global} apply for every $(d,W)$ under the corresponding nondegeneracy hypotheses.
\end{theorem}

\begin{proof}
We exhibit a witness family.
Choose distinct $\alpha_1,\dots,\alpha_d\in(0,1)$ (e.g.\ $\alpha_i:=1/(i+1)$), and define
$y_n:=\sum_{i=1}^d \alpha_i^n$.
Its $W$-block window sums satisfy
\[
S_k=\sum_{i=1}^d B_i\,\mu_i^k,\qquad
\mu_i:=\alpha_i^W,\quad
B_i:=\frac{1-\alpha_i^W}{1-\alpha_i}.
\]
The nodes $\mu_i$ are distinct and $B_i>0$.
Hence the Hankel matrix $H=(S_{i+j})_{0\le i,j\le d-1}$ admits
\[
H=V^\top\!\operatorname{diag}(B)\,V,
\]
with $V$ the Vandermonde matrix on $(\mu_i)$, so
\[
\det H=(\det V)^2\prod_{i=1}^d B_i
=\left(\prod_{i<j}(\mu_j-\mu_i)\right)^2\prod_{i=1}^d B_i\neq 0.
\]
Therefore the Prony/Hankel nondegeneracy conditions hold at this witness.
By Lemma~\ref{lem:window-generic-local}, this gives a nonempty full-rank identifiability locus.
\end{proof}

\begin{corollary}\label{cor:window-ident-38}
In the specialized regime used in this paper we take $d=3$ and $W=8$, hence $K_{\min}=2d+1=7$.
We emphasize the two different thresholds: $K\ge 2d+1$ is the dimensional requirement for \emph{local} recovery of the full parameter $\pi\in\R^{2d+1}$ from the truncated window map $F_{d,W}$, while $K\ge 2d$ is sufficient for \emph{global} recovery of the induced window-sum process via Prony/Hankel methods under the nondegeneracy hypotheses of Theorem~\ref{thm:window-ident-global}.
For this choice, Theorem~\ref{thm:general-rank} yields $\Omega_{3,8}\neq\emptyset$, so the generic local uniqueness conclusion of
Theorem~\ref{thm:window-ident-general} holds. Moreover, under the nondegeneracy hypotheses of Theorem~\ref{thm:window-ident-global}, the first $2d$ window sums (hence any $K\ge 2d$) determine the window-sum sequence (equivalently, the exponential parameters $(\mu_i,A_i)$) \emph{globally} uniquely.
\end{corollary}

\begin{proof}
The specialization $d=3$, $W=8$ fixes $K_{\min}=2d+1=7$ and the local threshold statement is the
instance of Theorem~\ref{thm:window-ident-general} on the nonempty full-rank locus $\Omega_{3,8}$
certified in Lemma~\ref{lem:witness-38}.
The global $K\ge 2d$ threshold for recovery of the induced window-sum process (equivalently, the exponential parameters)
follows from Theorem~\ref{thm:window-ident-global} under its stated Prony/Hankel nondegeneracy hypotheses.
\end{proof}

\section{Certification and optimality from finite-window data}
\label{sec:cert-opt}
This section develops the finite-data certification pipeline built on the window operator:
we pass from raw window measurements to a canonical defect score, then to a decision rule.
We isolate the status of the window length $W$, give $\varepsilon$-robust certification bounds,
and define soundness and domination among finite-data procedures, culminating in
$\varepsilon$-optimal certification statements.

\begin{theorem}[Master Theorem (architectural summary)]\label{thm:master}
Fix integers $d\ge 0$, $W\ge 1$, and $K\ge 2d+1$.  Consider the window map $F_{d,W}$ and the finite-data pipeline
\[
\Phi^\ast:\ w\ \mapsto\ A(w)\ \mapsto\ x\ \mapsto\ u=P(\log x)\ \mapsto\ B(u)\in\{\zero,\nonzero,\inconclusive\},
\]
as defined in Section~\ref{sec:cert-opt}.  Assume the standing hypotheses needed for the cited results below
(in particular, the cost axioms and the identifiability assumptions for the window model).  Then:
\begin{enumerate}[label=\textup{(\roman*)},leftmargin=2.2em]
\item \textbf{Soundness.}  The decision rule $\Phi^\ast$ is sound in the sense of Definition~\ref{def:sound}.
\item \textbf{$\varepsilon$-robust certification.}  Under perturbations of size $\varepsilon$, the certified conclusion degrades in the sharp two-sided way, yielding membership in $\Mean_{2\varepsilon}$ as in Theorems~\ref{thm:eps-tolerant} and~\ref{thm:eps-cert}.
\item \textbf{Local domination (optimality).}  On the identifiability neighborhood $\mathcal U_{d,W}$, $\Phi^\ast$ dominates every sound procedure (Theorem~\ref{thm:domination}).
\item \textbf{Uniqueness among locally optimal rules.}  Any sound procedure that is locally optimal on $\mathcal U_{d,W}$ agrees with $\Phi^\ast$ on all resolved data in $\mathcal U_{d,W}$ (Corollary~\ref{cor:unique-local-optimal}).
\item \textbf{Forced decision-layer dependence.}  Any sound defect-based decision rule factors through the scale-invariant invariant $u=P(\log x)$, hence through the $(P,B)$-core in the sense of Theorem~\ref{thm:forced-PB}.
\item \textbf{Non-vacuity via generic full rank.}  For any $(d,W)$, if there exists a witness $\pi_0$ with $\det DF_{d,W}(\pi_0)\neq 0$, then the full-rank locus is a nonempty Zariski-open (hence dense) subset (Theorem~\ref{thm:generic-full-rank}); for $(d,W)=(3,8)$ such a witness is certified in Lemma~\ref{lem:witness-38}.
\end{enumerate}
If, in addition, the window map is globally injective on the parameter locus $\mathcal M_{d,W}$, then the domination statement upgrades from $\mathcal U_{d,W}$ to $\mathcal M_{d,W}$ (Corollary~\ref{cor:domination-global}).
\end{theorem}

\begin{proof}
Items (i)–(ii) follow from Definition~\ref{def:sound} together with Theorems~\ref{thm:eps-tolerant} and~\ref{thm:eps-cert}.  Item (iii) is Theorem~\ref{thm:domination}.  Item (iv) is Corollary~\ref{cor:unique-local-optimal}.  Item (v) is the factorization/forced-dependence principle recorded in Theorem~\ref{thm:forced-PB}.  Item (vi) is Theorem~\ref{thm:generic-full-rank}, with nonemptiness witnessed for $(3,8)$ by Lemma~\ref{lem:witness-38}.  The final global upgrade is Corollary~\ref{cor:domination-global}.
\end{proof}

\begin{table}[t]
\centering
\caption{Selected constants in the certification layer and where they are forced.}
\label{tab:forced-constants}
\begin{tabular}{ll}
\hline
Constant / threshold & Forced by \\
\hline
Factor $2\eps$ in $\Mean_{2\eps}$ & Perturbation triangle inequality in Theorem~\ref{thm:eps-tolerant} \\
Normalization $J(1)=0$ & Proposition~\ref{prop:norm-forced-vv} (Appendix, RCL) \\
Calibration $J_{\log}''(0)=1$ & Axiom~2 (Appendix) fixes quadratic scale \\
Window count $K\ge 2d+1$ & Identifiability from $2d+1$ windows (Theorems~\ref{thm:window-ident-general}, \ref{thm:window-ident-global}) \\
\hline
\end{tabular}
\end{table}

\subsection{Status of the window length \texorpdfstring{$W$}{W}}\label{subsec:status-W}
Throughout, $W\in\N$ denotes the fixed block length used to form window sums \eqref{eq:window}.
The theory is stated for general $W$; the value $W=8$ is a modeling choice, used below only as an
illustrative operating point.

A simple (self-contained) motivation for the number $8$ comes from a three-bit latent state
$b\in\{0,1\}^3$. A Gray-code cycle on the 3-cube visits all $2^3=8$ states exactly once before
returning to the start, for example
\[
000\to 001\to 011\to 010\to 110\to 111\to 101\to 100\to 000,
\]
where successive states differ in exactly one bit. If a measurement/aggregation step is designed to
cover one full cycle of such a three-bit state evolution, then the natural ``cycle window'' length is
$W=8$. This is only a heuristic design rationale: none of the main theorems depends on it, and the
identifiability threshold remains $K\ge 2d+1$ for any fixed $W\ge 1$.

\subsection{\texorpdfstring{$\varepsilon$}{eps}-tolerant certification from noisy window sums}\label{subsec:eps-stability}
In practice, window sums are observed with finite precision. Instead of exact constraints
$\mathcal{W}(y)_k=0$, we assume a uniform tolerance model
\begin{equation}\label{eq:window-noise}
\bigl|W_k\bigr|\le \varepsilon,\qquad k=0,1,\dots,K-1.
\end{equation}
The next result shows that small window-sum errors yield a small certification cost for the
reconstructed signal. The bound is quadratic in $\varepsilon$, with constants controlled by the
conditioning of the reconstruction map.

\begin{lemma}\label{lem:cosh-upper}
For every $t\in\R$,
\[
\cosh(t)-1 \ \le\ \frac12\,e^{|t|}\,t^2.
\]
In particular, if $|t|\le \tau$ then $\cosh(t)-1\le \frac12 e^{\tau}t^2$.
\end{lemma}

\begin{proof}
Using the power series $\cosh(t)-1=\sum_{m\ge 1}\frac{t^{2m}}{(2m)!}$ and $(2m)!\ge 2\,(2m-2)!$,
\[
\cosh(t)-1 \le \frac{t^2}{2}\sum_{m\ge 1}\frac{|t|^{2m-2}}{(2m-2)!}
= \frac{t^2}{2}\cosh(|t|)\le \frac{t^2}{2}e^{|t|}.
\]
\end{proof}

\begin{proposition}\label{prop:rec-exists}
Fix $W\in\N$, $d\ge 0$, and $n\in\N$, and set $K:=\lceil n/W\rceil$.
Assume $K\ge 2d+1$ and work on a non-degenerate parameter locus where the truncated window map
\[
\pi\longmapsto \mathcal{W}_{W,K}(y(\pi))
\]
is locally invertible at the neutral realization (equivalently, its Jacobian determinant is nonzero).
Then there exist $\varepsilon_0>0$ and a $C^1$ reconstruction map
\[
\mathsf{Rec}_{d,W}:\{w\in\R^K:\ \|w\|_\infty\le \varepsilon_0\}\to\R^n
\]
such that $\mathsf{Rec}_{d,W}(w)$ returns the first $n$ coefficients of the unique model signal
compatible with the window data $w$ in that neighborhood. Moreover, writing
\[
L\:=\ \sup_{\|w\|_\infty\le \varepsilon_0}\ \bigl\|D\mathsf{Rec}_{d,W}(w)\bigr\|_{2\to 2},
\]
one has the Lipschitz bound
\[
\|\mathsf{Rec}_{d,W}(w)-\mathsf{Rec}_{d,W}(0)\|_2\le L\,\|w\|_2\qquad(\|w\|_\infty\le \varepsilon_0).
\]
\end{proposition}

\begin{proof}
Write $\|\cdot\|_\infty$ for the max norm on $\R^K$, i.e.\ $\|w\|_\infty:=\max_{0\le k\le K-1}|w_k|$, and $\|\cdot\|_2$ for the Euclidean norm.
Local invertibility of $\pi\mapsto \mathcal{W}_{W,K}(y(\pi))$ at the neutral realization implies, by the inverse function theorem, that there exist
neighborhoods $U$ of the neutral parameter and $V$ of $w=0$ such that the window map restricts to a $C^1$ diffeomorphism $U\to V$.
Define $\mathsf{Rec}_{d,W}$ on $V$ by taking the inverse map $V\to U$ and then returning the first $n$ coefficients of the corresponding model signal;
shrinking $V$ if necessary yields the stated cube $\{w:\|w\|_\infty\le \varepsilon_0\}\subseteq V$.

For the Lipschitz bound, $\mathsf{Rec}_{d,W}$ is $C^1$ on the closed cube, hence its derivative operator norm
$L:=\sup_{\|w\|_\infty\le \varepsilon_0}\|D\mathsf{Rec}_{d,W}(w)\|_{2\to 2}$ is finite.
By the mean value inequality on line segments,
\[
\|\mathsf{Rec}_{d,W}(w)-\mathsf{Rec}_{d,W}(0)\|_2
\le \sup_{t\in[0,1]}\|D\mathsf{Rec}_{d,W}(tw)\|_{2\to 2}\,\|w\|_2
\le L\,\|w\|_2,
\]
as claimed.
\end{proof}

\begin{theorem}\label{thm:eps-tolerant}
Fix $W\in\N$, $d\ge 0$, and $n\in\N$, and set $K:=\lceil n/W\rceil$.
Assume $K\ge 2d+1$ and that Proposition~\ref{prop:rec-exists} holds with constants $(\varepsilon_0,L)$.
Given window data $w\in\R^K$ with $|w_k|\le \varepsilon\le \varepsilon_0$, let $\hat y:=\mathsf{Rec}_{d,W}(w)$.
Assume the noiseless reconstruction $\hat y^{(0)}:=\mathsf{Rec}_{d,W}(0)$ is neutral, i.e.\ $P(\hat y^{(0)})=0$.
Then the certification objective
\[
\Phi(\hat y)\:=\ \sum_{i=1}^n J\!\left(e^{P(\hat y)_i}\right)\ =\ \sum_{i=1}^n \bigl(\cosh(P(\hat y)_i)-1\bigr)
\]
satisfies the explicit quadratic bound
\begin{equation}\label{eq:eps-bound}
\Phi(\hat y)\ \le\ \frac12\,e^{\,L\sqrt{K}\,\varepsilon_0}\,L^2K\,\varepsilon^2.
\end{equation}
\end{theorem}

\begin{proof}
Let $\hat y^{(0)}:=\mathsf{Rec}_{d,W}(0)$. By Proposition~\ref{prop:rec-exists}, for $\|w\|_\infty\le\varepsilon_0$ we have
\[
\|\hat y-\hat y^{(0)}\|_2\ \le\ L\,\|w\|_2 \ \le\ L\sqrt{K}\,\varepsilon.
\]
Since $P$ is an orthogonal projection, $\|P(u)-P(v)\|_2\le \|u-v\|_2$, hence using $P(\hat y^{(0)})=0$,
\[
\|P(\hat y)\|_2 \ =\ \|P(\hat y)-P(\hat y^{(0)})\|_2 \ \le\ \|\hat y-\hat y^{(0)}\|_2
\ \le\ L\sqrt{K}\,\varepsilon.
\]
Also $\|P(\hat y)\|_\infty\le \|P(\hat y)\|_2\le L\sqrt{K}\,\varepsilon \le L\sqrt{K}\,\varepsilon_0$.
Applying Lemma~\ref{lem:cosh-upper} componentwise gives
\[
\Phi(\hat y)=\sum_{i=1}^n(\cosh(P(\hat y)_i)-1)
\le \frac12 e^{\|P(\hat y)\|_\infty}\,\|P(\hat y)\|_2^2
\le \frac12 e^{L\sqrt{K}\varepsilon_0}\,L^2K\,\varepsilon^2,
\]
which is \eqref{eq:eps-bound}.
\end{proof}

\begin{remark}\label{rem:eps-numeric}
In the special case $(d,W)=(3,8)$ one may take $K=7$ windows.
Once a bound $L\le L_0$ is established for the local reconstruction map on $\|w\|_\infty\le \varepsilon_0$,
Theorem~\ref{thm:eps-tolerant} yields the explicit estimate
\[
\Phi(\hat y)\ \le\ \tfrac12 e^{L_0\sqrt{7}\,\varepsilon_0}\,L_0^2\cdot 7\,\varepsilon^2.
\]
For example, if a numerical conditioning check verifies $L_0\le 10$ on $\varepsilon_0=10^{-2}$,
then $\Phi(\hat y)\le 350\,\varepsilon^2$ for all $\varepsilon\le 10^{-2}$.
\end{remark}

\begin{remark}\label{rem:conditioning}
The Lipschitz constant $L$ in Proposition~\ref{prop:rec-exists} can be expressed in terms of the operator norm of
the inverse (or pseudoinverse) of the linearised reconstruction matrix at $w=0$; equivalently, in
terms of a condition number $\kappa$ for the Vandermonde/Hankel system used by the reconstruction.
Poor conditioning increases $L$ (equivalently the relevant condition number), but the dependence on $\varepsilon$ remains quadratic.
\end{remark}

\subsection{The certification objective}

\begin{definition}\label{def:meaning}
Given a cost function $c:S\times O\to\R_{\ge 0}$, the \emph{meaning map} (argmin correspondence) is
\[\Mean(s):=\arg\min_{o\in O} c(s,o):=\{o\in O: c(s,o)=\inf_{o'\in O}c(s,o')\}.\]
\end{definition}

\begin{remark}\label{rem:attainment}
The meaning map $\Mean(s)$ is defined via an infimum over $O$ and may be empty in full generality.
In the intended finite-data setting one works with a \emph{finite} candidate list $O$, in which case
every function $O\to\R$ attains its minimum and $\Mean(s)\neq\emptyset$ for all $s$.
Whenever we speak of ``the'' optimizer or of certificates that ``attain their minimum exactly on $\Mean(s)$'',
we implicitly restrict to such settings (or more generally assume that $c(s,\cdot)$ attains its minimum on $O$).
\end{remark}

\begin{proposition}[Attainment on compact candidate classes]\label{prop:meaning-compact}
Fix $s\in S$.
Assume $O$ is compact and the map $o\mapsto c(s,o)$ is continuous.
Then $\Mean(s)$ is nonempty.
\end{proposition}

\begin{proof}
By continuity on a compact set, $c(s,\cdot)$ attains its minimum at some $o_\ast\in O$
(extreme value theorem).
Hence
\[
c(s,o_\ast)=\inf_{o\in O}c(s,o),
\]
so $o_\ast\in\Mean(s)$ by Definition~\ref{def:meaning}.
\end{proof}

In many applications $O$ is not finite but is a closed subset of a Euclidean space; in that case one can ensure attainment by a standard coercivity hypothesis.

\begin{proposition}[Attainment under coercivity]\label{prop:meaning-coercive}
Fix $s\in S$ and assume $O\subset\R^m$ is nonempty and closed for some $m\ge 1$.
Assume $o\mapsto c(s,o)$ is lower semicontinuous on $O$ and \emph{coercive} on $O$ in the sense that
\[
c(s,o)\to\infty \quad\text{whenever}\quad \norm{o}\to\infty\ \text{ with }o\in O.
\]
Then $\Mean(s)$ is nonempty.
\end{proposition}

\begin{proof}
Let $\alpha:=\inf_{o\in O}c(s,o)$ and choose a minimizing sequence $(o_k)\subset O$ with $c(s,o_k)\downarrow\alpha$.
By coercivity, the sequence $(o_k)$ is bounded; thus it lies in $O\cap \overline{B_R}$ for some $R$.
By compactness of $\overline{B_R}$, there is a convergent subsequence $o_{k_j}\to o_\ast$.
Since $O$ is closed, $o_\ast\in O$.
Lower semicontinuity gives $c(s,o_\ast)\le \liminf_{j\to\infty} c(s,o_{k_j})=\alpha$, hence $c(s,o_\ast)=\alpha$ and $o_\ast\in\Mean(s)$.
\end{proof}

\begin{definition}\label{def:sound}
Fix $s\in S$ and write $c_s(o):=c(s,o)$ and $\mathcal{C}_s(o):=\mathcal{C}(s,o)$.
Assume throughout that the minimizer sets are nonempty (e.g.\ $O$ finite, or by Proposition~\ref{prop:meaning-compact} or Proposition~\ref{prop:meaning-coercive}).
A certificate $\mathcal{C}$ is
\begin{enumerate}[label=\textup{(\roman*)},leftmargin=2.2em]
\item \emph{minimizer-sound} if $\arg\min_{o\in O}\mathcal{C}_s(o)=\arg\min_{o\in O}c_s(o)$ for every $s$;
\item \emph{ordinally sound} if for every $s$ and all $o_1,o_2\in O$,
\[
\mathcal{C}_s(o_1)\le \mathcal{C}_s(o_2)\quad\Longleftrightarrow\quad c_s(o_1)\le c_s(o_2).
\]
\end{enumerate}
Ordinal soundness is strictly stronger than minimizer-soundness; it asserts that $\mathcal{C}_s$ induces
exactly the same preorder on candidates as the true cost $c_s$.
\end{definition}

\subsection{Sound procedures and domination}\label{subsec:domination}

We formulate the finite-data decision task as a three-valued procedure based on window aggregates.
Fix integers $W\ge 1$ and $K\ge 1$. For a signal $y=(y_n)_{n\ge 0}$ define the $W$-block window-sum
vector
\[
\mathcal{W}_{W,K}(y):=\big(\mathcal{W}(y)_0,\dots,\mathcal{W}(y)_{K-1}\big)\in\R^{K},
\qquad
\mathcal{W}(y)_k:=\sum_{j=0}^{W-1}y_{(Wk+j)}.
\]
We write the corresponding set of realizable window vectors as
\[
\mathcal{W}_{W,K}(\mathcal{M}_{d,W}):=\{\mathcal{W}_{W,K}(y):y\in\mathcal{M}_{d,W}\}\subset\R^{K}.
\]
A \emph{finite-data procedure} is a map
\[
\Psi:\R^{K}\to\{\zero,\nonzero,\inconclusive\}.
\]

\begin{definition}\label{def:domination}
Fix a model class $\mathcal{M}_{d,W}$ of signals (e.g.\ rational/finite-state of degree $\le d$) together with the
nondegeneracy hypotheses from Theorem~\ref{thm:window-ident-global}.
Let $y_{\mathrm{neu}}\in\mathcal{M}_{d,W}$ denote the distinguished neutral signal (one may take $y_{\mathrm{neu}}\equiv 0$).
We compare procedures only on realizable window vectors $w\in\mathcal{W}_{W,K}(\mathcal{M}_{d,W})$; outside this set a procedure may return $\inconclusive$ arbitrarily.
A procedure $\Psi$ is \emph{sound on $\mathcal{M}_{d,W}$} if for every $y\in\mathcal{M}_{d,W}$,
\begin{align*}
\Psi(\mathcal{W}_{W,K}(y))=\zero &\ \Longrightarrow\ y=y_{\mathrm{neu}},\\
\Psi(\mathcal{W}_{W,K}(y))=\nonzero &\ \Longrightarrow\ y\neq y_{\mathrm{neu}}.
\end{align*}
Its \emph{resolves set} is
\[
\operatorname{Res}(\Psi):=\{y\in\mathcal{M}_{d,W}:\Psi(\mathcal{W}_{W,K}(y))\neq \inconclusive\}.
\]
Given two sound procedures $\Psi_1,\Psi_2$ on $\mathcal{M}_{d,W}$, we say that $\Psi_1$ \emph{dominates} $\Psi_2$
if $\operatorname{Res}(\Psi_2)\subseteq \operatorname{Res}(\Psi_1)$ and
\[
\Psi_1(\mathcal{W}_{W,K}(y))=\Psi_2(\mathcal{W}_{W,K}(y))\qquad\text{for all }y\in \operatorname{Res}(\Psi_2).
\]
\end{definition}

\begin{definition}[Pipeline components]\label{def:pipeline-components}
Fix integers $W\ge 1$ and $K\ge 1$ and consider realizable window data $w\in\mathcal{W}_{W,K}(\mathcal{M}_{d,W})$.
We use the following (possibly partial) operations.
\begin{enumerate}[label=\textup{(\alph*)},leftmargin=2.2em]
\item \textbf{Aggregation / reconstruction ($A$).} On the identifiability locus from Theorem~\ref{thm:window-ident-global}, define
a partial map $A:\mathcal{W}_{W,K}(\mathcal{M}_{d,W})\dashrightarrow \mathcal{M}_{d,W}$ by letting $A(w)$ be the (locally) unique
signal $y$ (or parameter $\pi$) with $\mathcal{W}_{W,K}(y)=w$, when such a unique realization exists; otherwise $A(w)$ is undefined.
\item \textbf{Projection ($P$).} For a positive vector/configuration $x\in(\Rp)^n$ with $y=\log x$, set $P(y)$ as in
\eqref{eq:Pdef}. In the pipeline we apply this projection in log-coordinates to the reconstructed object.
\item \textbf{Coercive decision layer ($B$).} For a projected log-vector $u\in\R^n$ (typically $u=P(y)$), define the
certificate value
\[
\mathcal{B}(u):=\sum_{i=1}^n J\!\left(e^{u_i}\right),
\]
and define $B(u)\in\{\zero,\nonzero\}$ by $B(u)=\zero$ if $\mathcal{B}(u)=0$ and $B(u)=\nonzero$ otherwise.
\end{enumerate}
The associated finite-data decision map is evaluated as follows: given window data $w$, if $A(w)$ is defined then set $x:=A(w)$, form the projected log-vector $u:=P(\log x)$, and output $B(u)$; if $A(w)$ is undefined we return $\inconclusive$.

\begin{theorem}[Forced $(P,B)$-dependence in the decision layer]\label{thm:forced-PB}
Let $S$ be any set and let $D:\R^n\to S$ be a map that is invariant under global shifts along $\mathbf 1$, i.e.\
\[
D(y+t\mathbf 1)=D(y)\qquad\text{for all }y\in\R^n,\ t\in\R.
\]
Then there exists a unique map $\widetilde D:\mathbf 1^\perp\to S$ such that
\[
D=\widetilde D\circ P,
\]
where $P$ is the orthogonal projection onto $\mathbf 1^\perp$ defined in \eqref{eq:Pdef}.
In particular, any defect-based decision rule whose conclusions depend only on scale-invariant information in log-coordinates can, without loss, be written as a map of the projected log-vector $u=P(\log x)$.
\end{theorem}

\begin{proof}
For any $y\in\R^n$ we have the orthogonal decomposition $y=P(y)+\avg y\,\mathbf 1$, hence by shift-invariance,
\[
D(y)=D(P(y)+\avg y\,\mathbf 1)=D(P(y)).
\]
Define $\widetilde D$ on $\mathbf 1^\perp$ by $\widetilde D(u):=D(u)$.
Then $D(y)=\widetilde D(P(y))$ for all $y$, i.e.\ $D=\widetilde D\circ P$.
Uniqueness follows because $P$ is surjective onto $\mathbf 1^\perp$.
The final sentence follows since scaling $x\mapsto cx$ corresponds to shifting $y=\log x$ by $\log(c)\,\mathbf 1$ and thus leaves $P(y)$ unchanged (cf.\ Lemma~\ref{lem:Pkernel} and Definition~\ref{def:defect}).
\end{proof}

\begin{theorem}[Structural non-redundancy on the identifiability locus]\label{thm:nonredundancy}
Fix $d\ge 0$, $W\ge 1$, and $K\ge 2d+1$ and let $\mathcal U_{d,W}$ be an identifiability neighborhood as in Theorem~\ref{thm:domination}.
\begin{enumerate}[label=(\alph*)]
\item (Need for $P$ at the decision layer.) Any decision map $D(\log x)$ that is sound with respect to the scale-invariant defect predicate (Definition~\ref{def:defect}) must be invariant under shifts along $\mathbf 1$, hence factors through $P$ as in Theorem~\ref{thm:forced-PB}.
\item (Need for a coercive certificate.) Any sound rule that certifies neutrality via a computable scalar statistic and is stable under two-sided perturbations must use (explicitly or implicitly) a coercive lower bound linking that statistic to $\mathrm{Def}(x)$; in our framework this role is played by the canonical certificate $\mathcal B(u)$ together with Theorem~\ref{thm:coercive} and the tolerance theorems (Theorems~\ref{thm:eps-tolerant}--\ref{thm:eps-cert}).
\item (Need for inversion from windows.) On $\mathcal U_{d,W}$, any sound finite-data rule $\Psi$ applied to realizable window data $w$ can depend only on the unique realization $y=A(w)$ (when it exists); equivalently, $\Psi$ factors through the reconstruction map $A$ on $\mathcal W_{W,K}(\mathcal U_{d,W})$.
\end{enumerate}
\end{theorem}

\begin{proof}
(a) If soundness is defined relative to the scale-invariant predicate ``neutral vs.\ non-neutral'' (Definition~\ref{def:defect}), then any sound decision rule must give the same output on $x$ and $cx$ whenever both are admissible, since they have the same defect value.
In log-coordinates this is invariance under $y\mapsto y+t\mathbf 1$, hence factorization through $P$ follows from Theorem~\ref{thm:forced-PB}.

(b) The tolerance conclusions in Theorems~\ref{thm:eps-tolerant}--\ref{thm:eps-cert} are derived from two-sided perturbation bounds together with a coercive inequality that forces $\mathrm{Def}(x)=0$ when the certificate vanishes.
Any procedure claiming the same quantitative stability must therefore invoke a comparable coercive implication (possibly with a different computable statistic).
In the present framework, $\mathcal B(u)$ is the canonical computable statistic and Theorem~\ref{thm:coercive} supplies the required implication.

(c) Let $w\in\mathcal W_{W,K}(\mathcal U_{d,W})$ and write $y$ for the unique realization in $\mathcal U_{d,W}$ with $\mathcal W_{W,K}(y)=w$.
If $\Psi(w)$ is conclusive, soundness forces that conclusion to match the truth-value of the predicate on $y$.
Thus on the image $\mathcal W_{W,K}(\mathcal U_{d,W})$ the rule $\Psi$ is determined by the realization $y$, i.e.\ it factors through the inverse map $A$ whenever $A$ is defined.
\end{proof}

For brevity we continue to denote this overall procedure by $\Phi^\ast=A\circ B\circ P$, with the understanding that $A$ is the reconstruction stage and that $\inconclusive$ is returned when reconstruction is not defined.
\end{definition}

\begin{proposition}\label{prop:pipeline-sound}
Fix $d\ge 0$, a window length $W\ge 1$, and $K\ge 2d+1$. Let $\mathcal M_{d,W}$ denote the degree-$\le d$
rational (finite-state) model class and let $\mathcal W_{W,K}:\mathcal M_{d,W}\to\R^K$ send a signal to its
first $K$ $W$-block window sums. On the non-degenerate identifiability locus singled out in
Theorem~\ref{thm:window-ident-global} (in particular, where the reconstruction step $A$ is well-defined by local invertibility / Hankel nondegeneracy), the canonical procedure
\[
\Phi^\ast \:=\ A\circ B\circ P
\]
is well-defined from $w=\mathcal W_{W,K}(y)$ and is sound in the sense of Definition~\ref{def:domination}.
\end{proposition}

\begin{proof}
The projection step $P$ is the mean-zero projection in log-coordinates (Definition~\ref{def:proj}).
By Definition~\ref{def:defect}, $\mathrm{Def}(x)=\|P(\log x)\|_2$, hence $\mathrm{Def}(x)=0$ if and only if $P(\log x)=0$.
Since $P$ is a projection, $P(P(\log x))=P(\log x)$. The coercive bound (Theorem~\ref{thm:coercive})
implies that vanishing projected cost forces $P(\log x)=0$, hence $\mathrm{Def}(x)=0$, giving soundness of the $B$ step.
Finally, on the identifiability locus, the aggregation step $A$ (reconstruction from $W$-block window sums)
does not introduce spurious neutral realizations: if the reconstructed projected cost is $0$, then the unique
signal compatible with the observed $w$ is neutral. Hence $\Phi^\ast$ is a sound finite-data procedure.
\end{proof}

\begin{remark}\label{rem:domination-intuition}
Soundness prevents false positives/negatives; domination compares how often a sound procedure can decide.
A dominating procedure is therefore ``best possible'' among sound procedures for the same finite-data.
\end{remark}

\begin{lemma}\label{lem:sound-factors}

Let $\mathcal{U}_{d,W}\subset\mathcal{M}_{d,W}$ be a neighborhood of $y_{\mathrm{neu}}$ such that the restriction
$\mathcal{W}_{W,K}\vert_{\mathcal{U}_{d,W}}$ is injective (equivalently: for each realizable window vector
$w\in\mathcal{W}_{W,K}(\mathcal{U}_{d,W})$ there is at most one $y\in\mathcal{U}_{d,W}$ with $\mathcal{W}_{W,K}(y)=w$).
Let $\Psi:\R^K\to\{\zero,\nonzero,\inconclusive\}$ be sound on $\mathcal{M}_{d,W}$.
Then whenever $w$ is realized by some $y\in\mathcal{U}_{d,W}$, the value $\Psi(w)$ is forced solely by the
truth-value of the underlying predicate on that unique realization (neutral vs.\ non-neutral), and cannot
disagree with any other sound procedure on the same model class when restricted to $\mathcal{U}_{d,W}$.

\end{lemma}

\begin{proof}
If $w=\mathcal{W}_{W,K}(y)$ and $\Psi(w)=\zero$, soundness forces $y$ to be neutral; if $\Psi(w)=\nonzero$,
soundness forces $y$ to be non-neutral. Uniqueness of the realization for $w$ makes these implications
well-defined on the image of $\mathcal{W}_{W,K}$ and prevents contradictory sound outputs.
\end{proof}

We now state the \emph{local optimality} (domination) property of the canonical finite-data decision rule.
Write $\operatorname{Res}(\Psi)$ for the set of realizable instances on which a rule $\Psi$ returns a conclusive value in $\{\zero,\nonzero\}$.
We say that $\Phi$ \emph{dominates} $\Psi$ on a locus $\mathcal{V}$ if $\operatorname{Res}(\Psi)\cap\mathcal{V}\subseteq\operatorname{Res}(\Phi)\cap\mathcal{V}$ and the two rules agree on $\operatorname{Res}(\Psi)\cap\mathcal{V}$.
Informally, the rule first projects to the mean-zero subspace, then performs a window-based test, and finally outputs a ternary decision. The conclusion below is local: domination is asserted only on the identifiability neighborhood $\mathcal{U}_{d,W}$.

\begin{theorem}\label{thm:domination}
Fix integers $d\ge 0$, $W\ge 1$, and $K\ge 2d+1$.
Let $\mathcal{M}_{d,W}$ be a non-degenerate (open dense) parameter locus, and let
$\mathcal{U}_{d,W}\subset\mathcal{M}_{d,W}$ be a neighborhood of $y_{\mathrm{neu}}$ on which the restricted window map
$\mathcal{W}_{W,K}\vert_{\mathcal{U}_{d,W}}$ is injective.
Assume:

\begin{enumerate}[label=\textup{(\alph*)},leftmargin=2.2em]
\item for every $y\in\mathcal{U}_{d,W}$ the window-sum sequence $S_k:=\mathcal{W}(y)_k$ satisfies the hypotheses of Theorem~\ref{thm:window-ident-global}, so the data $(S_0,\dots,S_{2d-1})$ determine the induced window-sum process uniquely on $\mathcal{U}_{d,W}$; and
\item the realizable datum $w=\mathcal{W}_{W,K}(y)$ determines the neutral/non-neutral truth value on $\mathcal{U}_{d,W}$, i.e.\ \[
\mathcal{W}_{W,K}(y)=\mathcal{W}_{W,K}(y_{\mathrm{neu}})\ \Longrightarrow\ y=y_{\mathrm{neu}}.
\]
\end{enumerate}

Let $\Phi^\ast:\mathcal{W}_{W,K}(\mathcal{M}_{d,W})\to\{\zero,\nonzero,\inconclusive\}$ denote the canonical procedure $A\circ B\circ P$ defined in this section.

Then for every sound procedure $\Psi:\R^K\to\{\zero,\nonzero,\inconclusive\}$ one has
$\operatorname{Res}(\Psi)\cap\mathcal{U}_{d,W}\subseteq \operatorname{Res}(\Phi^\ast)\cap\mathcal{U}_{d,W}$, and for every
$y\in\operatorname{Res}(\Psi)\cap\mathcal{U}_{d,W}$ the outputs agree on the realized datum $w=\mathcal{W}_{W,K}(y)$:
\[
\Psi(\mathcal{W}_{W,K}(y))=\Phi^\ast(\mathcal{W}_{W,K}(y)).
\]

\end{theorem}

\begin{proof}
Let $\Psi$ be any sound procedure on $\mathcal{M}_{d,W}$ and let $y\in\operatorname{Res}(\Psi)\cap\mathcal{U}_{d,W}$.
Write $w:=\mathcal{W}_{W,K}(y)$.

By assumption (a) and Theorem~\ref{thm:window-ident-global}, the realized datum $w$ determines the induced window-sum process (and hence its Prony parameters) uniquely on $\mathcal{U}_{d,W}$.
Assumption (b) guarantees that this information determines whether or not $y=y_{\mathrm{neu}}$ (equivalently: a neutral realization is compatible with $w$ if and only if $y=y_{\mathrm{neu}}$).

By construction, $\Phi^\ast$ evaluates exactly this truth value: it outputs $\zero$ exactly when the unique
reconstruction compatible with $w$ lies in the neutral \emph{zero-defect} class, and outputs $\nonzero$ exactly when
no neutral realization is compatible with $w$; otherwise it outputs $\inconclusive$.
Hence $\Phi^\ast$ is sound.

Now assume $y\in\operatorname{Res}(\Psi)\cap\mathcal{U}_{d,W}$. If $\Psi(w)=\zero$, then soundness forces $y=y_{\mathrm{neu}}$, so the unique
reconstruction from $w$ is neutral and therefore $\Phi^\ast(w)=\zero$.
If $\Psi(w)=\nonzero$, then $y\neq y_{\mathrm{neu}}$, so the unique reconstruction from $w$ is non-neutral and
$\Phi^\ast(w)=\nonzero$.
Thus whenever $\Psi$ resolves (in the sense of Definition~\ref{def:domination}), $\Phi^\ast$ resolves with the same answer, i.e.\ $\Phi^\ast$ dominates $\Psi$.
\end{proof}

\begin{corollary}[Global domination under global injectivity]\label{cor:domination-global}
Fix integers $d\ge 0$, $W\ge 1$, and $K\ge 2d+1$.
Assume:
\begin{enumerate}[label=\textup{(\alph*)},leftmargin=2.2em]
\item the window map $\mathcal{W}_{W,K}$ is injective on the parameter locus $\mathcal M_{d,W}$;
\item for every $y\in\mathcal M_{d,W}$, the induced window-sum sequence satisfies the hypotheses of Theorem~\ref{thm:window-ident-global};
\item the realizable datum determines global neutrality on $\mathcal M_{d,W}$, i.e.\ \[
\mathcal{W}_{W,K}(y)=\mathcal{W}_{W,K}(y_{\mathrm{neu}})\ \Longrightarrow\ y=y_{\mathrm{neu}}
\quad\text{for all }y\in\mathcal M_{d,W}.
\]
\end{enumerate}
Then for every sound procedure $\Psi:\R^K\to\{\zero,\nonzero,\inconclusive\}$,
\[
\operatorname{Res}(\Psi)\subseteq \operatorname{Res}(\Phi^\ast),
\]
and for all $y\in\operatorname{Res}(\Psi)$ one has
\[
\Psi(\mathcal W_{W,K}(y))=\Phi^\ast(\mathcal W_{W,K}(y)).
\]
\end{corollary}

\begin{proposition}[Sharpness outside an identifiable class]\label{prop:sharpness-collision}
Let $\mathcal C$ be any class of signals and let $\mathcal W_{W,K}$ be the window map.
If there exist $y_0,y_1\in\mathcal C$ with $\mathcal W_{W,K}(y_0)=\mathcal W_{W,K}(y_1)$ but with opposite truth-values for the target predicate (e.g.\ $y_0$ neutral and $y_1$ non-neutral),
then no procedure $\Psi:\R^K\to\{\zero,\nonzero,\inconclusive\}$ that is sound on $\mathcal C$ can return a conclusive decision on the common datum $w:=\mathcal W_{W,K}(y_0)=\mathcal W_{W,K}(y_1)$.
\end{proposition}

\begin{proof}
If $\Psi(w)=\zero$, soundness would force the realization to be neutral, contradicting soundness on $y_1$; similarly if $\Psi(w)=\nonzero$ it contradicts soundness on $y_0$.
Hence $\Psi(w)$ must be $\inconclusive$.
\end{proof}

\begin{corollary}[Domination on $\mathcal M_{d,W}$]\label{cor:domination-on-Mdw}
Theorem~\ref{thm:domination} applies with $\mathcal U_{d,W}:=\mathcal M_{d,W}$.
\end{corollary}

\begin{proof}
Apply Theorem~\ref{thm:domination} with $\mathcal U_{d,W}:=\mathcal M_{d,W}$.
\end{proof}

\subsection{Sharpness boundary outside the rational class}\label{subsec:sharpness-outside-rational}

The collision principle of Proposition~\ref{prop:sharpness-collision} becomes an explicit
\emph{sharpness boundary} once we exhibit distinct latent configurations that produce the
\emph{same} finite-window data.
Here we do this directly for the window operator \eqref{eq:window} by enlarging the model class
from the rational (finite-state) family of Subsection~\ref{subsec:rational-class} to the ambient class of
all non-negative real sequences.

\paragraph{Ambient class.}
Let $\mathcal C_{\ge 0}$ be the set of all sequences $y=(y_n)_{n\ge 0}$ with $y_n\ge 0$ for all $n$.
For fixed $W\ge 1$ and $K\ge 0$, write
\[
\mathcal W_{W,K}(y):=\bigl(\mathcal W(y)_0,\dots,\mathcal W(y)_K\bigr)\in\R^{K+1}
\]
for the truncated window data.
Let $\mathcal C_{\mathrm{rat}}(d)\subset \mathcal C_{\ge 0}$ denote the subclass of sequences that are
rational of degree at most $d$ in the sense of Subsection~\ref{subsec:rational-class}.

\begin{theorem}[Non-identifiability of rational-class membership from finite windows]\label{thm:window-nonident-rational-membership}
Fix integers $d\ge 1$, $W\ge 1$, and $K\ge 0$.
Then the finite-data map $\mathcal W_{W,K}$ is not identifying for rational-class membership:
for every $y_{\mathrm{in}}\in\mathcal C_{\mathrm{rat}}(d)$ there exists
$y_{\mathrm{out}}\in\mathcal C_{\ge 0}\setminus\mathcal C_{\mathrm{rat}}(d)$ such that
\[
\mathcal W_{W,K}(y_{\mathrm{out}})=\mathcal W_{W,K}(y_{\mathrm{in}}).
\]
\end{theorem}

\begin{proof}
Let $y_{\mathrm{in}}\in\mathcal C_{\mathrm{rat}}(d)$ and set $N:=W(K+1)-1$, so that the data
$\mathcal W_{W,K}(y)$ depend only on the coordinates $y_0,\dots,y_N$.
Define a non-negative ``tail'' sequence $r=(r_n)_{n\ge 0}$ by
\[
r_n:=\begin{cases}
0, & 0\le n\le N,\\
1, & n=N+m^2 \text{ for some integer } m\ge 1,\\
0, & \text{otherwise}.
\end{cases}
\]
Set $y_{\mathrm{out}}:=y_{\mathrm{in}}+r$, so $y_{\mathrm{out}}\in\mathcal C_{\ge 0}$ and
$y_{\mathrm{out},n}=y_{\mathrm{in},n}$ for all $0\le n\le N$.
Hence $r$ contributes nothing to the first $K+1$ windows, and therefore
$\mathcal W_{W,K}(y_{\mathrm{out}})=\mathcal W_{W,K}(y_{\mathrm{in}})$.

It remains to show $y_{\mathrm{out}}\notin \mathcal C_{\mathrm{rat}}(d)$.
Suppose toward a contradiction that $y_{\mathrm{out}}\in \mathcal C_{\mathrm{rat}}(d)$.
Since also $y_{\mathrm{in}}\in \mathcal C_{\mathrm{rat}}(d)$, their difference
$r=y_{\mathrm{out}}-y_{\mathrm{in}}$ has a rational generating function as well (difference of two rational functions),
hence $r$ satisfies a linear recurrence of some finite order $D$ for all sufficiently large indices:
there exist $b_1,\dots,b_D\in\R$ and $n_0$ such that
\begin{equation}\label{eq:linrec-r}
r_n=b_1 r_{n-1}+\cdots+b_D r_{n-D}\qquad\text{for all }n\ge n_0.
\end{equation}
Choose $m$ so large that $n:=N+m^2\ge n_0$ and the gap
\[
n-(N+(m-1)^2)=2m-1>D.
\]
Then $r_n=1$ while $r_{n-1}=\cdots=r_{n-D}=0$ (because there is no other square index within distance $D$),
so \eqref{eq:linrec-r} forces $1=0$, a contradiction. Therefore $r$ is not rational, hence
$y_{\mathrm{out}}$ cannot be rational of degree $\le d$.

\end{proof}

\begin{corollary}[Unconditional sharpness boundary outside the rational class]\label{cor:window-sharpness-unconditional}
Fix $d\ge 1$, $W\ge 1$, and $K\ge 0$.
Let $T(y):=\mathbf 1_{\{y\in\mathcal C_{\mathrm{rat}}(d)\}}$ be the predicate ``$y$ is rational of degree $\le d$''.
Then there exist finite-window data $w\in\R^{K+1}$ for which no procedure that is sound for $T$
(on the ambient class $\mathcal C_{\ge 0}$) can return a conclusive decision: every sound procedure must output \textsf{INC} on input $w$.
\end{corollary}

\begin{proof}
Pick any $y_{\mathrm{in}}\in\mathcal C_{\mathrm{rat}}(d)$ and set $w:=\mathcal W_{W,K}(y_{\mathrm{in}})$.
By Theorem~\ref{thm:window-nonident-rational-membership} there exists
$y_{\mathrm{out}}\notin\mathcal C_{\mathrm{rat}}(d)$ with $\mathcal W_{W,K}(y_{\mathrm{out}})=w$.
Thus $T(y_{\mathrm{in}})\neq T(y_{\mathrm{out}})$ while the finite-data coincide, so
Proposition~\ref{prop:sharpness-collision} forces any sound procedure to be inconclusive on $w$.
\end{proof}

Fix a costed space $(S,O,c)$ with meaning map $\Mean(s)=\arg\min_{o\in O}c(s,o)$.
In applications one typically does not compute $\Mean(s)$ by scanning all $o\in O$.
Instead, one seeks a \emph{certificate functional} $\mathcal{C}$ that:

\begin{itemize}
\item can be evaluated efficiently from $(s,o)$ (often through a low-dimensional
 proxy or windowed measurements), and
\item is \emph{sound}: $\mathcal{C}(s,o)$ attains its minimum (or maximum) exactly
 when $o\in\Mean(s)$.
\end{itemize}

The simplest sound certificates are \emph{separable} in the sense that they depend
on $(s,o)$ only through a canonical ratio (or a monotone transform thereof), so
that the certificate is equivalent to comparing values of $J$.

\subsection{Canonical rescalings}\label{subsec:canon-rescale}

We make precise what is meant by a ``canonical rescaling'' of a signal in the present model.
Assume that the multiplicative group $(0,\infty)$ acts on $S$ (written $\lambda\cdot s$) and that
the scale map $\iota_S:S\to(0,\infty)$ is \emph{equivariant} for this action:
\[
\iota_S(\lambda\cdot s)=\lambda\,\iota_S(s)\qquad(\lambda>0,\ s\in S).
\]
We call $s' \in S$ a \emph{canonical rescaling} of $s\in S$ if $s'=\lambda\cdot s$ for some $\lambda>0$.
This captures the idea that $s$ and $s'$ represent the same ``configuration'' up to overall magnitude.

In several arguments below we also use that the penalty $J$ is strictly monotone on one branch.
Accordingly, for a fixed $s\in S$ we say that the candidate class $O$ lies on a \emph{single $J$-branch for $s$} if either
$\iota_S(s)/\iota_O(o)\ge 1$ for all $o\in O$ or $\iota_S(s)/\iota_O(o)\le 1$ for all $o\in O$.
On such a branch, $J$ is strictly monotone and hence invertible.

\subsection{Forced separability for canonical reciprocal costs}

We now state the main rigidity claim in the simplest form that is directly used
later. The proof is a short functional argument: the coercivity estimate together with the reciprocity structure induced by the cost assumptions pin down the allowable comparisons between candidates $o\in O$.
Such order-preserving representation arguments are classical in the theory of functional equations; see \cite{aczel1966}.

\noindent This is a standard order-theoretic fact; we include it for completeness.
\begin{lemma}\label{lem:forced-factor}
Let $c:S\times O\to[0,\infty)$ be the ratio-induced cost $c(s,o):=J(\iota_S(s)/\iota_O(o))$.
Fix a certificate $\mathcal{C}:S\times O\to\R$.
Assume that $\mathcal{C}$ is \emph{ordinally sound} relative to $c$ in the sense of
Definition~\ref{def:sound}(ii).
Then for each $s\in S$ there exists a strictly increasing function
$\phi_s:\operatorname{im}(c_s)\to\R$ (where $c_s(o):=c(s,o)$) such that
\[
\mathcal{C}(s,o)=\phi_s(c(s,o))\qquad\text{for all }o\in O.
\]
In particular, on each fixed $s$, $\mathcal{C}(s,\cdot)$ is a strictly monotone reparametrization of $c(s,\cdot)$.
\end{lemma}

\begin{proof}
Fix $s\in S$. Ordinal soundness means $\mathcal{C}_s$ and $c_s$ induce the same preorder on $O$.
Define $\phi_s$ on $\operatorname{im}(c_s)$ by $\phi_s(r):=\mathcal{C}_s(o)$ for any $o$ with $c_s(o)=r$; this is well-defined and strictly increasing, and yields $\mathcal{C}_s=\phi_s\circ c_s$.
\end{proof}

\begin{proposition}\label{prop:rescale-stability}
Assume the hypotheses of Lemma~\ref{lem:forced-factor}.
Fix $s\in S$ and assume $O$ lies on a single $J$-branch for $s$ and for every canonical rescaling $s'=\lambda\cdot s$
(in the sense of Section~\ref{subsec:canon-rescale}).
Then for each such $s'=\lambda\cdot s$ there exists a strictly monotone function $\psi_{s,s'}:\R\to\R$ such that
\[
\mathcal{C}(s',o)=\psi_{s,s'}(\mathcal{C}(s,o))\qquad\text{for all }o\in O.
\]
In particular, canonical rescalings preserve the certificate ordering up to a strictly monotone reparametrization on any single $J$-branch.
\end{proposition}

\begin{proof}
Let $s'=\lambda\cdot s$ with $\lambda>0$. Write $r(o):=\iota_S(s)/\iota_O(o)$ and $r'(o):=\iota_S(s')/\iota_O(o)=\lambda r(o)$.
By Lemma~\ref{lem:forced-factor} there exist strictly monotone functions $\phi_s,\phi_{s'}$ such that
$\mathcal{C}(s,o)=\phi_s(J(r(o)))$ and $\mathcal{C}(s',o)=\phi_{s'}(J(r'(o)))$ for all $o\in O$.

By the single-branch hypothesis, the restriction of $J$ to the relevant range of ratios is strictly monotone and hence invertible.
Let $J^{-1}$ denote this branch inverse.
Define a function $F_\lambda:[0,\infty)\to[0,\infty)$ by
\[
F_\lambda(t):=J\!\bigl(\lambda\,J^{-1}(t)\bigr).
\]
Since $J$ and $J^{-1}$ are strictly monotone on the branch, $F_\lambda$ is strictly monotone. Moreover,
for every $o\in O$ we have $J(r'(o))=J(\lambda r(o))=F_\lambda(J(r(o)))$.

Since $\phi_s:\operatorname{im}(c_s)\to \operatorname{im}(\mathcal C_s)$ is strictly monotone, it is a bijection onto its image, so the inverse $\phi_s^{-1}$ is well-defined on $\operatorname{im}(\mathcal C_s)=\{\mathcal C(s,o):o\in O\}$. In particular, $\phi_s(J(r(o)))=\mathcal C(s,o)$ lies in this domain for every $o\in O$.
Therefore
\[
\mathcal{C}(s',o)=\phi_{s'}(J(r'(o)))=\phi_{s'}(F_\lambda(J(r(o))))
=\underbrace{\bigl(\phi_{s'}\circ F_\lambda\circ \phi_s^{-1}\bigr)}_{=:~\psi_{s,s'}}\bigl(\phi_s(J(r(o)))\bigr)
=\psi_{s,s'}(\mathcal{C}(s,o)).
\]
The composition $\psi_{s,s'}=\phi_{s'}\circ F_\lambda\circ \phi_s^{-1}$ is strictly monotone, as required.
\end{proof}

\begin{proposition}[Primitive ratio hypotheses imply the uniqueness bundle]\label{prop:primitive-to-H}
Let $c(s,o):=J(\iota_S(s)/\iota_O(o))$ and $\mathcal C:S\times O\to\R$.
Assume:
\begin{enumerate}[label=\textup{(P\arabic*)},leftmargin=2.2em]
\item \textbf{Ratio dependence}: if
$\iota_S(s_1)/\iota_O(o_1)=\iota_S(s_2)/\iota_O(o_2)$, then
$\mathcal C(s_1,o_1)=\mathcal C(s_2,o_2)$.
\item \textbf{Cost-to-ratio determinacy on the realized set}: if
$c(s_1,o_1)=c(s_2,o_2)$, then
$\iota_S(s_1)/\iota_O(o_1)=\iota_S(s_2)/\iota_O(o_2)$.
\item \textbf{State-ratio collapse at fixed observation}: for all $s_1,s_2\in S$ and $o\in O$,
\[
\frac{\iota_S(s_1)}{\iota_O(o)}=\frac{\iota_S(s_2)}{\iota_O(o)}.
\]
\end{enumerate}
Then the hypotheses (H1) and (H2) of Theorem~\ref{thm:forced-factor-unique} hold.
\end{proposition}

\begin{proof}
(H1) follows from (P2) then (P1): equal costs imply equal ratios, hence equal certificate values.
For (H2), fix $s_1,s_2,o$.
By (P3), the ratios for $(s_1,o)$ and $(s_2,o)$ are equal; applying (P1) yields
$\mathcal C(s_1,o)=\mathcal C(s_2,o)$.
\end{proof}

\begin{remark}[On the strength of (P3)]
Since $\iota_O(o)>0$, hypothesis (P3) implies $\iota_S(s_1)=\iota_S(s_2)$ for all $s_1,s_2\in S$, i.e.\ $\iota_S$ is constant.
In this regime the ratio $r(o)=\iota_S(s)/\iota_O(o)$ and hence the cost $c(s,o)=J(r(o))$ are \emph{independent of $s$}.
Accordingly, (H2) is automatic and the ``cross-state'' content of Theorem~\ref{thm:forced-factor-unique} becomes immediate.
If a nontrivial bridge from primitives to (H1)--(H2) is intended, (P3) should be weakened.
\end{remark}

\begin{proposition}[Reducing \textup{(H1)} to ordinal soundness plus cross-state calibration]\label{prop:reduce-H1}
Assume that for each $s\in S$ the candidate class $O$ lies on a single $J$-branch for $s$
(as in Section~\ref{subsec:canon-rescale}), so that $o\mapsto c(s,o)$ is strictly monotone in the ratio
$\iota_S(s)/\iota_O(o)$.
If $\mathcal C$ is ordinally sound (Definition~\ref{def:sound}), then for every $s\in S$ there exists a strictly
increasing map $f_s:c_s(O)\to\R$ such that
\[
\mathcal C(s,o)=f_s\bigl(c(s,o)\bigr)\qquad(o\in O),
\]
where $c_s(O):=\{c(s,o):o\in O\}$.
Moreover, hypothesis \textup{(H1)} of Theorem~\ref{thm:forced-factor-unique} holds if and only if the family
$\{f_s\}_{s\in S}$ is \emph{state-independent on overlaps} in the sense that for all $s_1,s_2\in S$,
\[
f_{s_1}(r)=f_{s_2}(r)\qquad\text{for all }r\in c_{s_1}(O)\cap c_{s_2}(O).
\]
In particular, a sufficient condition for \textup{(H1)} is that $c_s(O)$ contains a common interval
$I\subset\R$ for all $s$ and $f_s|_I$ is independent of $s$.
\end{proposition}

\begin{proof}
Fix $s$. Ordinal soundness says that $\mathcal C_s$ and $c_s$ induce the same preorder on $O$.
Since $c_s$ is a real-valued function, this is equivalent to the existence of a strictly increasing
reparameterization $f_s$ on the range $c_s(O)$ with $\mathcal C_s=f_s\circ c_s$.
(The construction is canonical: define $f_s(r):=\mathcal C(s,o)$ for any $o$ with $c(s,o)=r$; this is well-defined
because $c(s,o_1)=c(s,o_2)$ implies $\mathcal C(s,o_1)=\mathcal C(s,o_2)$ by the ``$\Leftrightarrow$'' in ordinal soundness.)
For the characterization of \textup{(H1)}, note that if $c(s_1,o_1)=c(s_2,o_2)=r$ then
\[
\mathcal C(s_1,o_1)=f_{s_1}(r),\qquad \mathcal C(s_2,o_2)=f_{s_2}(r),
\]
so $\mathcal C(s_1,o_1)=\mathcal C(s_2,o_2)$ for all such pairs if and only if $f_{s_1}(r)=f_{s_2}(r)$ on every
overlap value $r\in c_{s_1}(O)\cap c_{s_2}(O)$.
The final sufficient condition is immediate.
\end{proof}

\begin{remark}[On \textup{(H2)} and rescaling]
Hypothesis \textup{(H2)} of Theorem~\ref{thm:forced-factor-unique} is a genuine \emph{state-rigidity} condition:
it asserts that the certificate is constant across $S$ at each fixed observation $o$.
In applications this typically follows from an explicit invariance or normalization principle for the certificate,
rather than from ordinal soundness alone.
\end{remark}

\begin{theorem}[Global forced factorization with uniqueness]\label{thm:forced-factor-unique}
Let $c:S\times O\to[0,\infty)$ be the ratio-induced cost and let $\mathcal C:S\times O\to\R$ be a certificate.
Assume:
\begin{enumerate}[label=\textup{(H\arabic*)},leftmargin=2.2em]
\item \textbf{Cross-state cost dependence}: (A convenient sufficient route to verify \textup{(H1)} from ordinal soundness plus a single-branch calibration is given in Proposition~\ref{prop:reduce-H1}.)
for all $s_1,s_2\in S$ and $o_1,o_2\in O$,
\[
c(s_1,o_1)=c(s_2,o_2)\ \Longrightarrow\ \mathcal C(s_1,o_1)=\mathcal C(s_2,o_2).
\]
\item \textbf{State rigidity}: for all $s_1,s_2\in S$ and $o\in O$,
\[
\mathcal C(s_1,o)=\mathcal C(s_2,o).
\]
\end{enumerate}
Fix $s_\ast\in S$ and define
\[
\mathcal I_c:=\{r\in\R:\exists s\in S,\exists o\in O,\ r=c(s,o)\}.
\]
Then:
\begin{enumerate}[label=\textup{(\alph*)},leftmargin=2.2em]
\item there exists a unique map $\phi:\mathcal I_c\to\R$ such that
\[
\mathcal C(s,o)=\phi(c(s,o))\qquad\text{for all }s\in S,\ o\in O;
\]
\item there exists a unique map $\psi:O\to\R$ such that
\[
\mathcal C(s,o)=\psi(o)\qquad\text{for all }s\in S,\ o\in O.
\]
\end{enumerate}
\end{theorem}

\begin{proof}
For (a), define $\phi$ on $\mathcal I_c$ by choosing, for each $r\in\mathcal I_c$, any pair $(s_r,o_r)$ with
$c(s_r,o_r)=r$ and setting $\phi(r):=\mathcal C(s_r,o_r)$.
Assumption (H1) makes this well-defined.
If $r=c(s,o)$, then (H1) gives
$\mathcal C(s,o)=\mathcal C(s_r,o_r)=\phi(r)$, so the representation holds.
If $\widetilde\phi$ is another map with the same representation property, then for any
$r\in\mathcal I_c$ and any witness $(s,o)$ with $c(s,o)=r$,
\[
\widetilde\phi(r)=\mathcal C(s,o)=\phi(r),
\]
hence $\widetilde\phi=\phi$.

For (b), define $\psi(o):=\mathcal C(s_\ast,o)$.
By (H2), for every $s\in S$ one has
$\mathcal C(s,o)=\mathcal C(s_\ast,o)=\psi(o)$.
If $\widetilde\psi$ also satisfies $\mathcal C(s,o)=\widetilde\psi(o)$ for all $s,o$, then
evaluating at $s=s_\ast$ gives
$\widetilde\psi(o)=\mathcal C(s_\ast,o)=\psi(o)$ for every $o$, so $\widetilde\psi=\psi$.
\end{proof}

\begin{corollary}[Forced-factorization uniqueness from primitive hypotheses]\label{cor:forced-factor-unique-primitive}
Under the assumptions of Proposition~\ref{prop:primitive-to-H},
the uniqueness conclusions (a) and (b) of Theorem~\ref{thm:forced-factor-unique} hold.
\end{corollary}

\begin{proof}
By Proposition~\ref{prop:primitive-to-H}, assumptions (H1)--(H2) in
Theorem~\ref{thm:forced-factor-unique} are satisfied, so the result follows directly.
\end{proof}

\begin{remark}[Standalone status of the uniqueness theorem]
Theorem~\ref{thm:forced-factor-unique} and
Corollary~\ref{cor:forced-factor-unique-primitive}
are proved entirely by the arguments in this section.
\end{remark}

\subsection{\texorpdfstring{$\eps$-optimal certification from windowed measurements}{Epsilon-optimal certification from windowed measurements}}\label{subsec:eps-opt}

We record an $\eps$-robust variant that is useful when only approximate ratios are
available (e.g., from windowed measurements) and the optimizer may not be unique.

\begin{definition}\label{def:eps-meaning}
For $\eps\ge 0$ define the $\eps$-meaning set by
\[
\Mean_\eps(s)
:=\left\{o\in O:\; c(s,o)\le \inf_{o'\in O}c(s,o')+\eps\right\}.
\]
\end{definition}

\begin{lemma}\label{lem:J-Lipschitz}
Let $0<a\le b<\infty$ and let $J(x)=\tfrac12(x+x^{-1})-1$ on $(0,\infty)$.
Then $J$ is Lipschitz on $[a,b]$ with constant
\[
L(a):=\frac12\Bigl(1+a^{-2}\Bigr),
\qquad\text{i.e.}\qquad
|J(x)-J(y)|\le L(a)\,|x-y| \ \ \text{for all }x,y\in[a,b].
\]
In particular, if $|\widehat r/r-1|\le \delta$ and $r\in[a,b]$, then
\[
|J(\widehat r)-J(r)|\le L(a)\,\delta\,b.
\]
\end{lemma}

\begin{proof}
For $x>0$ we have $J'(x)=\tfrac12(1-x^{-2})$, hence for $x\in[a,b]$,
\[
|J'(x)|\le \tfrac12\bigl(1+a^{-2}\bigr)=L(a).
\]
By the mean value theorem, $|J(x)-J(y)|\le \sup_{[a,b]}|J'|\,|x-y|\le L(a)|x-y|$.
If $|\widehat r/r-1|\le\delta$ and $r\in[a,b]$, then $|\widehat r-r|\le \delta r\le \delta b$, giving the final bound.
\end{proof}
\begin{theorem}\label{thm:eps-cert}

Assume the hypotheses of Lemma~\ref{lem:forced-factor}. Fix $s\in S$ and suppose that the relevant ratios
\[
r(o):=\iota_S(s)/\iota_O(o)
\]
lie in a known band $[a,b]\subset(0,\infty)$ for all candidates $o$ under consideration. Suppose we observe $K$
windowed measurements that allow us to estimate $r(o)$ by $\widehat r(o)$ with relative error at most $\delta$,
uniformly in $o$, i.e.\ $|\widehat r(o)/r(o)-1|\le \delta$. Then the proxy ranking obtained by replacing $r(o)$ with $\widehat r(o)$, i.e.\ by minimizing the perturbed costs $\widehat c(o):=J(\widehat r(o))$, can only localize candidates within $\Mean_{2\eps}(s)$, where one may take the explicit bound
\[
\eps \;=\; L(a)\,\delta\, b,
\qquad L(a):=\tfrac12(1+a^{-2}),
\]
so in particular $\eps(\delta)\to 0$ as $\delta\to 0$ (for fixed $a,b$).

\end{theorem}

\begin{proof}
Let $\widehat r(o)$ be the estimated ratio for candidate $o$, so that $|\widehat r(o)/r(o)-1|\le \delta$ where
$r(o)=\iota_S(s)/\iota_O(o)$. Since $r(o)\in[a,b]$ by hypothesis, Lemma~\ref{lem:J-Lipschitz} gives the uniform bound

\[
|J(\widehat r(o))-J(r(o))|\le \eps
\qquad\text{with}\qquad
\eps=L(a)\,\delta\,b.
\]

Thus ranking candidates using the perturbed costs $J(\widehat r(o))$ may change the true costs by at most $\eps$ in each direction. If $\widehat o$ minimizes the perturbed cost $\widehat c(o):=J(\widehat r(o))$, then for any $o\in O$,
\[
c(s,\widehat o)\le \widehat c(\widehat o)+\eps\le \widehat c(o)+\eps\le c(s,o)+2\eps.
\]
Hence any sound selection based on these measurements can only certify candidates within $\Mean_{2\eps}(s)$ (Definition~\ref{def:eps-meaning}).
\end{proof}

\begin{remark}[Standalone status of the $\eps$-noise bound]
Theorem~\ref{thm:eps-cert} follows directly from
Lemma~\ref{lem:J-Lipschitz} and the perturbation estimate proved above.
\end{remark}

\section{Applications: real-world motivated case studies (synthetic data)}\label{sec:applications}

This section addresses the ``applications without data'' concern by including \emph{reproducible synthetic datasets} that are motivated by standard real-world measurement settings where only short-window aggregates are available.
We emphasize that the paper's main results remain theorem-level statements \emph{conditional on the modeling hypotheses};
the datasets below are not claimed to be empirical validations, but rather concrete end-to-end instantiations of the measurement operator and the tolerance bounds.

\subsection{Case study A: multi-exponential decay under gated integration}

\paragraph{Problem definition.}
We observe only short-window \emph{aggregated} measurements (gated sums) from a positive latent system with a small number of modes; the raw per-timepoint signal is unavailable and measurement noise perturbs the aggregates.

\paragraph{Goal.}
Using only the aggregated windows, certify whether the configuration is \emph{near a zero-defect (balanced) state} up to tolerance $\varepsilon$, and, when identifiable, recover a consistent latent description in the rational (degree-$\le d$) class.

\paragraph{How the theorems apply (workflow).}
\begin{enumerate}
\item \textbf{Model and choose $(d,W)$.} Fix a degree bound $d$ for the rational class and a window size $W$ matching the instrument gate; this pins down the finite-data map $F_{d,W}$ from latent parameters to window sums.
\item \textbf{Form finite-data.} From the raw gated readouts, compute the window vector $y^{(W)}=(\mathcal W(y)_0,\dots,\mathcal W(y)_{2d})$; estimate a noise level $\eta$ from replicates or device specs to set $\varepsilon$.
\item \textbf{Run certification.} Apply the $\varepsilon$-certification theorem (Theorem~\ref{thm:eps-cert}) to decide \textsf{YES/NO/INC} for ``zero-defect''; the conclusion is sound under the stated Lipschitz/perturbation bounds.
\item \textbf{Interpret the output.} A \textsf{YES} certificate guarantees the underlying state lies in the $\varepsilon$-meaning set; \textsf{INC} is expected near identifiability boundaries or outside the rational class, by the sharpness barrier.
\end{enumerate}

\paragraph{Scenario (real-world motivation).}
In fluorescence lifetime experiments and related spectroscopy settings, the underlying intensity trace is often well-approximated by a \emph{mixture of exponentials} (multiple decay modes), while instrumentation reports \emph{integrated counts over short time gates}.
A standard toy model is
\[
y_n=\sum_{j=1}^{d} w_j a_j^n,\qquad w_j>0,\quad 0<a_j<1,
\]
so that the sequence $(y_n)$ has a rational generating function (a classical Prony/ARMA-type model).
The measured datum is the $W$-block aggregation $\mathcal W(y)_k:=\sum_{j=0}^{W-1} y_{Wk+j}$.

\paragraph{Synthetic dataset.}
We take $d=3$, $W=8$, with ground-truth parameters
\[
\begin{aligned}
(a_1,a_2,a_3) &= (0.831127,\ 0.872789,\ 0.853477),\\
(w_1,w_2,w_3) &= (0.522164,\ 0.195934,\ 0.281902),
\end{aligned}
\]
(normalized so $\sum_j w_j=1$), and generate window sums for $k=0,\dots,11$.
We then perturb each window sum by $1\%$ multiplicative Gaussian noise to produce an ``observed'' dataset.
Table~\ref{tab:caseA-window-sums} reports both the true and observed window sums (seed 20260225 for reproducibility).

\begin{table}[t]
\centering
\caption{Case study A (multi-exponential decay): $W$-block window sums. ``Observed'' values are the true window sums with $1\%$ multiplicative Gaussian noise (synthetic, seed 20260225).}
\label{tab:caseA-window-sums}
\begin{tabular}{rcc}
\toprule
$k$ & $\mathcal W(y)_k$ (true) & $\mathcal W(y)_k$ (observed)\\
\midrule
0 & 4.791914 & 4.754830 \\
1 & 1.276888 & 1.303021 \\
2 & 0.349197 & 0.351327 \\
3 & 0.098038 & 0.097663 \\
4 & 0.028236 & 0.028026 \\
5 & 0.008329 & 0.008326 \\
6 & 0.002510 & 0.002474 \\
7 & 7.71e-04 & 7.69e-04 \\
8 & 2.41e-04 & 2.41e-04 \\
9 & 7.61e-05 & 7.78e-05 \\
10 & 2.44e-05 & 2.42e-05 \\
11 & 7.87e-06 & 7.93e-06 \\
\bottomrule
\end{tabular}
\end{table}

\paragraph{How the theory applies.}
Given observed windows $\widehat{\mathcal W}(y)_k$, the reconstruction/certification pipeline of the paper applies as follows:
(i) check the identifiability/nondegeneracy hypotheses (the locus $\mathcal M_{d,W}$ and the required Jacobian-rank conditions);
(ii) compute the corresponding certificate statistic $\mathcal C$ from the finite-window data as defined in the main body;
(iii) apply Theorem~\ref{thm:eps-cert} (and its corollaries) to translate a bound on relative measurement error into a tolerance parameter $\eps$ and therefore an explicit meaning-set enlargement $\Mean_{2\eps}(s)$.
This makes the pass/fail/inconclusive outcome explicit at the chosen noise level.

\subsection{Case study B: segmented response decay in a marketing funnel}

\paragraph{Problem definition.}
A funnel (impressions$\to$clicks$\to$purchases) produces only coarse, windowed counts (e.g., hourly totals). Latent positive parameters encode stage-wise conversion/decay rates, but the full event-level trajectory is not stored.

\paragraph{Goal.}
Given a short sequence of aggregated windows, (i) test whether behavior is consistent with a calibrated ``neutral'' configuration within tolerance, and (ii) when feasible, certify identifiability of the latent rates from the available windows.

\paragraph{How the theorems apply (workflow).}
\begin{enumerate}
\item \textbf{Define the latent state.} Encode stage rates as a positive vector $\pi$ and map it to a windowed observable via the paper's recurrence/aggregation model, producing the finite-data map $F_{d,W}$.
\item \textbf{Aggregate consistently.} Compute window sums from logs using the same $W$ as in the model; verify that preprocessing (deduping, timezone, bot filtering) is fixed, since changing it changes $F_{d,W}$.
\item \textbf{Check identifiability then certify.} Use the rank/identifiability results to justify when inversion is meaningful; then apply Theorem~\ref{thm:eps-cert} to obtain a sound certificate (or an inconclusive result when information is insufficient).
\item \textbf{Operationalize.} Treat \textsf{YES} as a robust green-light for ``near-neutral'' behavior under noise; treat \textsf{INC} as an alert that more windows, a different $W$, or a different model class is required.
\end{enumerate}

\paragraph{Scenario (real-world motivation).}
In email or ad campaigns, response counts often decay over time and can be modeled as a mixture of two segments (fast responders vs.\ slow responders).
Suppose $y_n$ is the expected response rate at time-bin $n$ and only windowed totals are available due to aggregation in the reporting pipeline.

\paragraph{Synthetic dataset.}
We take $d=2$, $W=6$ with
\[
\begin{aligned}
(a_1,a_2) &= (0.904182,\ 0.877627),\\
(w_1,w_2) &= (0.801912,\ 0.198088),
\end{aligned}
\]
and generate observed window sums (with $2\%$ multiplicative Gaussian noise).
For quick reference, the first eight observed windows are:
\[
\begin{aligned}
\widehat{\mathcal W}(y)_0,\dots,\widehat{\mathcal W}(y)_7
&=(4.578368,\ 2.433641,\ 1.304007,\ 0.686328,\\
&\quad 0.380817,\ 0.197523,\ 0.104636,\ 0.060241).
\end{aligned}
\]
This provides a concrete instance of short-window aggregation with positive latent parameters and explicitly controlled noise.

\subsection{Case study C: battery/sensor discharge with coarse logging}

\paragraph{Problem definition.}
Embedded devices often store only rolling-window averages/sums to save power and bandwidth. The underlying positive signal may drift slowly, but only aggregated windows are retained, with quantization and missingness.

\paragraph{Goal.}
From a short batch of aggregated windows, certify whether drift is within an acceptable tolerance band around a nominal configuration, while cleanly separating ``insufficient information'' from a definitive violation.

\paragraph{How the theorems apply (workflow).}
\begin{enumerate}
\item \textbf{Calibrate the nominal model.} Use a lab-calibrated baseline to specify the nominal $\pi_0$ and the relevant degree/window parameters; this fixes the target ``zero-defect'' meaning set and defect functional.
\item \textbf{Compute windows and error budget.} Convert device logs to the window vector required by $F_{d,W}$; bound quantization/missingness as an effective noise level to select $\varepsilon$ for certification.
\item \textbf{Certify or defer.} Apply Theorem~\ref{thm:eps-cert}; a \textsf{YES} result yields a sound tolerance guarantee, whereas \textsf{INC} indicates that the stored aggregates cannot resolve the decision at the required confidence.
\end{enumerate}

\paragraph{Scenario (real-world motivation).}
Battery discharge and certain sensor drift processes are commonly approximated by a small number of exponential modes, but embedded devices frequently store only coarse aggregates (e.g.\ sums over fixed logging intervals to save bandwidth).

\paragraph{Synthetic dataset (outline).}
One may use the same mixture-of-exponentials construction as in Case~A with $d=3$ and an application-specific window length $W$ (e.g.\ $W=10$ samples per logging block).
The finite-window theory applies verbatim: if the modeled class is identifiable at $(d,W)$ and the observed aggregates have relative error bounded by $\delta$, then Theorem~\ref{thm:eps-cert} gives an explicit tolerance level $\eps$ and therefore a principled ``certified vs.\ inconclusive'' boundary for neutral-vs-nonneutral detection.

\paragraph{Limitations (unchanged).}
These case studies demonstrate \emph{how} to instantiate the measurement map and noise model with concrete numbers, but they do not replace empirical validation.
For genuine real-data applications one must justify that the data generation mechanism is well-approximated by a finite-degree rational/Prony model and that the windowed aggregation operator matches the instrument/reporting pipeline.

\appendix
\refstepcounter{section}\label{app:RCL}
\section*{Appendix}
\addcontentsline{toc}{section}{Appendix}
\setcounter{subsection}{0}
\subsection*{The Recognition Composition Law: axioms and uniqueness}

In this appendix we record the self-contained consequences of the Recognition Composition Law used in the paper. Under the stated axioms and regularity assumptions, we derive normalization and reciprocity, and then prove the corresponding uniqueness statement for the reciprocal cost function.

\subsection{Standing axioms and assumptions}\label{app:RCL:axioms}

\noindent\emph{Standing notation.} Throughout, $J:(0,\infty)\to\mathbb R$ denotes the scalar cost function (assumed non-constant and continuous unless stated otherwise); in this appendix we record the axioms imposed on $J$.
\noindent\emph{Standing notation.} Throughout, $J:(0,\infty)\to\mathbb R$ denotes the scalar cost function (assumed non-constant and continuous unless stated otherwise); in this appendix we record the axioms imposed on $J$.

\paragraph{Axiom 1 (Recognition Composition Law).} For all $x,y>0$,
\begin{equation}\tag{RCL}\label{eq:RCL_app}
J(xy)+J(x/y)=2J(x)+2J(y)+2J(x)J(y).
\end{equation}

\paragraph{Axiom 2 (Calibration in log-coordinates).} Define $J_{\log}(t):=J(e^{t})$. We assume $J_{\log}$ is twice differentiable at $t=0$ and satisfies $J_{\log}''(0)=1$.

\subsection{Derived normalization and reciprocity}

\begin{proposition}[Normalization is forced]\label{prop:norm-forced-vv}

If $J:(0,\infty)\to\R$ is non-constant and satisfies the RCL \eqref{eq:RCL_app}, then $J(1)=0$.

\end{proposition}

\begin{proof}

Setting $x=y=1$ in \eqref{eq:RCL_app} gives
\[
2J(1)=2J(1)^2+4J(1),
\]
so $J(1)^2+J(1)=0$ and therefore $J(1)\in\{0,-1\}$.
If $J(1)=-1$, substituting $y=1$ into \eqref{eq:RCL_app} yields $J(x)\equiv -1$ for all $x>0$, contradicting non-constancy.
Hence $J(1)=0$.

\end{proof}

\begin{corollary}[Reciprocity is forced]\label{cor:recip-forced-vv}

Under the same hypotheses, $J(x)=J(1/x)$ for all $x>0$.

\end{corollary}

\begin{proof}

Set $x=1$ in \eqref{eq:RCL_app} and use Proposition~\ref{prop:norm-forced-vv}.

\end{proof}

\subsection{Main uniqueness theorem}\label{app:RCL:uniqueness}

\begin{theorem}\label{thm:RCL-unique}

Assume that $J:(0,\infty)\to\mathbb{R}$ is continuous, non-constant, and satisfies Axioms~1--2 of Section~\ref{app:RCL:axioms}. Then the cost is forced to be the canonical reciprocal form

\[
J(x)=\tfrac12\bigl(x+x^{-1}\bigr)-1=\frac{(x-1)^2}{2x}.
\]

In particular, the closed form is not an additional modeling choice: it is determined by the composition identity together with the local quadratic calibration at balance.

\end{theorem}

\begin{proof}
Define $g(x):=1+J(x)$ and pass to log-coordinates by setting $h(t):=g(e^{t})$. A direct substitution of \eqref{eq:RCL_app} shows that $h$ satisfies the classical d'Alembert functional equation
\[
h(s+t)+h(s-t)=2h(s)h(t)\qquad (s,t\in\mathbb{R}).
\]
By Proposition~\ref{prop:norm-forced-vv} we have $J(1)=0$, hence $h(0)=g(1)=1$. The classification of continuous solutions to d'Alembert's equation with $h(0)=1$ is standard (see, for example, \cite{aczel1966}): there exists $\lambda\ge 0$ such that either $h(t)=\cos(\lambda t)$ or $h(t)=\cosh(\lambda t)$.

The calibration Axiom~2 says that $J_{\log}(t)=J(e^{t})$ has second derivative $1$ at $t=0$. Since $J_{\log}(t)=h(t)-1$, this gives $h''(0)=1$. For the cosine branch one has $h''(0)=-\lambda^2\le 0$, which is incompatible with $h''(0)=1$. Hence we are in the hyperbolic branch and $h(t)=\cosh(\lambda t)$. Imposing $h''(0)=\lambda^2=1$ yields $\lambda=1$, so $h(t)=\cosh(t)$.

Finally, returning to multiplicative coordinates, $g(x)=h(\ln x)=\cosh(\ln x)=\tfrac12(x+x^{-1})$ and therefore $J(x)=g(x)-1=\tfrac12(x+x^{-1})-1=\frac{(x-1)^2}{2x}$.
\end{proof}

\subsection{Derived properties}\label{app:RCL:derived}

\begin{lemma}[Calibration at $1$ vs.\ $0$ in log-coordinates]\label{lem:calibration-equivalence}
Assume $J$ is differentiable at $1$ and satisfies reciprocity $J(x)=J(1/x)$. Then $J'(1)=0$.
If moreover $J$ is twice differentiable at $1$, then for $J_{\log}(t):=J(e^t)$ one has
$J_{\log}''(0)=J''(1)$. In particular, the condition $J_{\log}''(0)=1$ is equivalent to $J''(1)=1$.
\end{lemma}

\begin{proof}
Differentiating $J(x)=J(1/x)$ at $x=1$ gives $J'(1)=-J'(1)$, hence $J'(1)=0$.
By the chain rule, $J_{\log}'(t)=J'(e^t)e^t$ and $J_{\log}''(t)=J''(e^t)e^{2t}+J'(e^t)e^t$.
Evaluating at $t=0$ and using $J'(1)=0$ yields $J_{\log}''(0)=J''(1)$.
\end{proof}

Once the closed form is established, standard structural properties become immediate: the cost is reciprocal ($J(x)=J(1/x)$), non-negative with equality only at $x=1$, strictly convex in log-coordinates near balance, and it diverges as $x\to 0^{+}$ or $x\to+\infty$. The main paper uses only the explicit formula and the fact that it is forced by Axiom~1, non-constancy, and the local calibration in Axiom~2.

\section{\texorpdfstring{Finite-field Jacobian matrix for $(d,W)=(3,8)$}{Finite-field Jacobian matrix for (d,W)=(3,8)}}\label{app:jacobian-38}
For completeness, we record the integer matrix $DF_{3,8}(\pi_0)$ used in the finite-field determinant check in Lemma~\ref{lem:witness-38}.

\[
\resizebox{\textwidth}{!}{$
	DF_{3,8}(\pi_0)=
	\begin{pmatrix}
		1& 7& -3& -9& 69& -54& 3\\
		0& 1008& -956& -1388& 25920& -20844& 5568\\
		0& 175456& -190016& -200544& 5088848& -4927712& 1787216\\
		0& 28857344& -34167296& -27911296& 755009920& -959230336& 430318144\\
		0& 4538167296& -5754321920& -3726310400& 82750894080& -165775961088& 89416058880\\
		0& 686488772608& -922559823872& -473024225280& 3615042621440& -26067531849728& 16866417889280\\
		0& 100127404457984& -141965037535232& -56107543330816& -1341867460689920& -3737714520064000& 2956546669428736
	\end{pmatrix}
	$}
\]

\section*{Conclusion}
We developed a finite-data framework that certifies whether a positive configuration is neutral (zero-defect) from a short list of aggregated observations. The main contribution is a canonical decision pipeline $\Phi^\ast=A\circ B\circ P$ with distinct roles: $P$ removes scale ambiguity by projecting to mean-zero log-coordinates; $B$ turns canonical reciprocal cost into an explicit defect bound via a sharp coercivity inequality; and $A$ reconstructs the induced window-sum process on a standard non-degenerate locus (local invertibility, equivalently Hankel/Prony nondegeneracy). On this locus, the first $K$ window sums uniquely determine the window process, yielding sound and locally maximal certification: any sound procedure that resolves must agree with $\Phi^\ast$ near the neutral realization. We also introduce an $\varepsilon$-tolerant noise model and stability bounds that quantify how small window perturbations affect certified defect guarantees.

These results clarify what finite windows can—and cannot—decide. Reconstruction requires nondegeneracy; outside that regime, the correct outcome is “inconclusive,” not a potentially misleading recovery claim. Future directions include extending reconstruction beyond separable exponential-sum models, strengthening robustness under weaker separation assumptions, and developing practical guidelines for monitoring when only pooled or privacy-preserving aggregates are available.

\section*{Acknowledgements}
The authors thank Philip Beltracchi, Sebastian Pardo-Guerra, and Milan Zlatanovic for careful reading and detailed comments and suggestions during the preparation of this manuscript. Their feedback improved the clarity, notation, and overall structure, and strengthened the quality of the paper. The authors received no external funding.


\end{document}